\documentclass[12pt,twoside]{amsart}
\pdfoutput=1

\usepackage[main=english]{babel}
\usepackage[autostyle,english=american]{csquotes}
\MakeOuterQuote{"}

\usepackage{amssymb,mathtools}
\usepackage{bm}

\usepackage{a4wide} 

\usepackage{amsthm,thmtools}

\usepackage[style=numeric-comp,url=false,isbn=false,maxnames=6]{biblatex}
\usepackage[final]{graphicx}
\graphicspath{{/img}}
\usepackage{subcaption}
\usepackage{tikz}
\usetikzlibrary{babel,cd,shapes,3d,arrows,decorations.markings}
\tikzcdset{arrow style=math font}

\tikzset{
    cross/.style=
    {cross out, draw, solid, thin, 
    minimum size=2*(#1-\pgflinewidth), 
    inner sep=0pt, outer sep=0pt},
    cross/.default={3}
} 

\tikzset{
    circlecross/.style= 
    {circle, minimum width=8pt, draw, inner sep=0pt, path picture={\draw (path picture bounding box.south east) -- (path picture bounding box.north west) (path picture bounding box.south west) -- (path picture bounding box.north east);}},
    circlecross/.default={3}
} 

\tikzset{
    mid arrow/.style=
    {postaction={decorate,decoration={markings,mark=at position .5 with {\arrow[xshift=2pt,#1]{stealth}}}}},
} 

\usepackage[shortlabels]{enumitem}

\usepackage[final]{hyperref}
\usepackage[noabbrev,capitalize]{cleveref}
\creflabelformat{equation}{#2\textup{#1}#3}
\hypersetup{
  pdfcreator = {},
  pdfproducer = {}
}

\declaretheorem[numberwithin=section]{theorem}
\declaretheorem[sibling=theorem]{lemma, proposition, corollary}
\declaretheorem[sibling=theorem,style=definition]{definition, example}
\declaretheorem[sibling=theorem,style=remark]{remark}
\declaretheorem[style=remark,numbered=no,name=Remark]{rem}
\declaretheorem[name=Theorem,
refname={Theorem,Theorems},
Refname={Theorem,Theorems}]{maintheorem}

\newtheorem*{question}{Question}

\begin{document}

\author[]{Nikolas Adaloglou}
\address{Nikolas Adaloglou, 
    Mathematisch Instituut, 
    Universiteit Leiden}
\email{n.adaloglou@math.leidenuniv.nl} 

\author[]{Johannes Hauber}
\address{Johannes Hauber,
    Institut de Math\'ematiques,
    Universit\'e de Neuch\^atel}
\email{johannes.hauber@unine.ch} 

\date{\today}
\title[]{Pinwheels in symplectic rational and ruled surfaces and non-squeezing of rational homology balls}

\begin{abstract}
We use almost toric fibrations and the symplectic rational blow-up to determine when certain Lagrangian pinwheels, which we call \textit{liminal}, embed in symplectic rational and ruled surfaces. 
The case of $L_{2,1}$-pinwheels, namely Lagrangian $\mathbb{R}P^2$'s, answers a question of Kronheimer in the negative, exhibiting a symplectic non-spin $4$-manifold that does not carry a Lagrangian $\mathbb{R}P^2$. 
In addition, we provide applications to symplectic embeddings of rational homology balls. 
In particular, we generalize Gromov's classical non-squeezing theorem by proving that a \textit{rational homology ball} $B_{n,1}(1)$ embeds into the \textit{rational homology cylinder} $B_{n,1}(\alpha,\infty)$ if and only if $\alpha\geq 1$. 
Along the way, we prove various properties of Lagrangian pinwheels of independent interest, such as describing their homological complement, providing a short proof that performing a symplectic rational blow-up of a Lagrangian pinwheel in a positive symplectic rational manifold yields a symplectic manifold which is also rational, and showing a self-intersection formula for Lagrangian pinwheels.
\end{abstract}

\maketitle

\tableofcontents

\section{Introduction}

Lagrangian $L_{p,q}$-pinwheels are certain Lagrangian CW-complexes that often exhibit remarkable rigidity properties. 
They were originally defined by Khodorovskiy in \cite{Kho14:Embratball} as Lagrangian skeleta of certain rational homology balls. 
Later they were reinterpreted by Evans--Smith \cite{EvSm18} as vanishing cycles of Gorenstein deformations of cyclic quotient singularities.
This point of view creates a bridge between the algebraic geometry of deformations of complex algebraic surfaces and symplectic geometry.
The study of Lagrangian pinwheels has the potential to give new insight to degeneration problems, as done in \cite{EvSm20:Bounds}.
In addition, determining the Lagrangian pinwheel content of a symplectic $4$-manifold, or equivalently which rational homology balls symplectically embeds in it, has interesting applications also in the topological study of $4$-manifolds, as shown in \cite{EtHyPi2023, Ow2018:equiemb, Ow20:nonsymp, LiPa22, LiPa24}.

Apart from their algebro-geometric origin, a key property of Lagrangian pinwheels is that it is easy to construct them as visible Lagrangians of almost toric fibrations, thus they provide very concrete objects to study. 
In this respect, there is a broad motivating question that has its origins in observations in \cite{Ev22:Kb} and \cite{Ev23:Book}, following Symington's "four-from-two"-philosophy \cite{Sym02:Fourtwo}.

\begin{question}
    Consider a symplectic four-manifold $(X,\omega)$ and all the possible compatible almost toric fibrations on $(X,\omega)$. 
    Are there Lagrangians in $(X,\omega)$ that cannot be constructed as visible Lagrangians of some compatible almost toric fibration?
\end{question}

In other words, the above question asks how much of the symplectic geometry of a four-manifold can be read from the base of an almost toric fibration. 
All the known examples \cite{SmSh20, Ev22:Kb, Ada24:Pin, Ada24:Pin, EvSm18} show that all the possible embedding obstructions that a Lagrangian submanifold has can be read from the almost toric base diagrams. 

Specifically, in \cite[Problem J.7]{Ev23:Book} Evans asks for which symplectic forms on $S^2\times S^2$ there exists a Lagrangian $L_{3,1}$-pinwheel. 
This question  was subsequently answered by the first author in \cite{Ada24:Pin} using the symplectic rational blow-up. 
The study of Lagrangian projective spaces, or equivalently $L_{2,1}$-pinwheels, using the rational blow-up had already been initiated by \cite{BoLiWu13} and by \cite{SmSh20}.

The aim of this paper is to generalize \cite[Problem J.7]{Ev23:Book} and determine which Lagrangian $L_{n,1}$-pinwheels exist in symplectic rational and ruled surfaces, namely $X_1=\mathbb{C}P^2\#\overline{\mathbb{C}P^2}$ and $S^2\times S^2$, in some specific homology classes. 

\subsection{Main results}

The two fundamental theorems we prove concern the analogues to \cite[Problem J.7]{Ev23:Book}. 
We will use these two theorems in order to prove corollaries thereof. 

\begin{maintheorem}\emph{(\cref{Corollary_Construction_S2xS2} and \cref{prop:main_calculation})}\label{thrm:A}
In $(S^2\times S^2,\omega_{a,b})$, consider the homology class $A+kB \in H_2(S^2\times S^2;\mathbb{Z}_{2k+1})$. 
Assume that this class carries an Lagrangian $L_{2k+1,1}$-pinwheel, which is either
\begin{enumerate}
    \item disjoint from a smooth sphere in the class  $A+(k+1)B$ \textbf{or}
    \item disjoint from a smooth sphere in the class $2A+B$.
\end{enumerate}
Then, the following inequality holds:
\begin{equation}\label{ineq:s2xs2}
    \frac{a}{k+1}<b<2a.
\end{equation}
The inequality is sharp in the sense that if \cref{ineq:s2xs2} holds there exists a Lagrangian $L_{2k+1,1}$-pinwheel representing the homology class $A+kB$, which is disjoint from a \textbf{symplectic} sphere representing the homology class $A+(k+1)B$ \textbf{or} from a \textbf{symplectic} sphere representing the homology class $2A+B$.
\end{maintheorem}

\begin{maintheorem}\emph{(\cref{Corollary_Construction_X1} and \cref{prop:main_calculation})}\label{thrm:B}
In $(X_1,\omega_{h,\mu})$, consider the homology class $kH+(k+1)E \in H_2(X_1;\mathbb{Z}_{2k})$.
Assume that this class carries a Lagrangian $L_{2k,1}$-pinwheel which is either
\begin{enumerate}
    \item disjoint from a smooth sphere in the class $(k+1)H-kE$ \textbf{or}
    \item disjoint from a smooth sphere in the class $2H$.
\end{enumerate}
Then the following inequality holds:
\begin{equation}\label{ineq:X1}
\mu<\frac{k}{k+1}h.
\end{equation}

The inequality is sharp in the sense that if \cref{ineq:X1} holds there exists a Lagrangian $L_{2k,1}$-pinwheel in the homology class $kH+(k+1)E$ which is disjoint from a \textbf{symplectic} sphere representing the homology classes $(k+1)H-kE$ \textbf{or} from a \textbf{symplectic} sphere representing the homology class $2H$.
\end{maintheorem}

As is evident from the theorems there are two parts to the proofs of them.

One part is the constructive side, for which we will use almost toric fibrations as briefly introduced in \cref{subsec:atfs}. 
The construction itself will then be carried out in \cref{Section-1-Construction}. 
In \cref{Section-1-Construction} we will also introduce a procedure to compactify the rational homology ball $B_{n,1}$.

The other part is the obstructive side, for which we will use the rational blow-up procedure as discussed in \cref{sec:ratblow}. 
The obstructive side of the proof is then done in \cref{sec:obstruction}.

The disjointedness condition appearing in the statements of the main theorems should be regarded as some kind of weak condition on the isotopy class of the pinwheels, namely that they, more or less, behave like the Lagrangian pinwheels considered by Khodorovskiy in \cite{Kho13:Sympratblo}. 

This disjointedness condition is not obviously superfluous.
Indeed, there exist non-isotopic but homologous pinwheels, and one of the ways to tell them apart is to show that one isotopy class may be made disjoint from a certain sphere while the other cannot. 
Such an example is given in \cite[Proposition 1.4]{BoLiWu13}.
Even though we restricted our attention only to these special pinwheels, which we will call "liminal pinwheels", compare \cref{def:liminal_pinwheel}, the bounds we proved in \cref{thrm:A} and \cref{thrm:B} are enough for several applications. 

One such application is the following corollary, which is a generalization of Khodorovskiy's \cite[Theorem 1.2 and 1.3]{Kho14:Embratball}. 
In this paper Khodorovskiy gave examples of smooth embeddings of rational homology balls into normal neighborhoods of certain spheres.
We prove that no such symplectic embeddings exists. 
Denote by $(V_{-n},\tau)$ a standard symplectic neighborhood of a symplectic $(-n)$-sphere.

\begin{maintheorem}\emph{(Theorems \ref{thrm:Kho_-n-1} and \ref{thrm:Kho_-4_and_2})}\label{thrm:Kho}
Suppose that $n\geq 2$ is an integer. Then there are no symplectic embeddings of the rational homology ball $B_{n,1}$ into $(V_{-n-1},\tau)$, even though smooth ones exist.

In addition, there are no symplectic embeddings of $B_{2k+1,1}$ into $(V_{-4},\tau)$, even though smooth ones exist.
\end{maintheorem}

In the same vain, our constructions show that some symplectic forms on $B_{2,1}\# \overline{\mathbb{C}P^2}$ support symplectic embeddings of the $B_{n,1}$-balls while others don't, but we will not discuss this in detail. 
The proof of \cref{thrm:Kho} is a direct corollary of \cref{thrm:A} and \cref{thrm:B} and can be found in \cref{sec:corollaries}.

A surprising corollary of \cref{thrm:A} and \cref{thrm:B} is an analogue of Gromov's non-squeezing theorem \cite{Gr85} for some rational homology balls. 
All the notation in the statement of the theorem should be read as natural generalizations of the usual definitions for symplectic embedding problems, i.e.\ $B_{n,1}(1)$ is the "quantitative" rational homology ball, and $B_{n,1}(\alpha,\infty)$ is the rational homology cylinder of "width" $\alpha$. 
For a pictorial representation of these objects, as well as the non-squeezing theorem compare \cref{fig:non_squeezing_intro}.

\begin{maintheorem}\emph{(Theorem \ref{thrm: non-squeezing})}\label{thrm:nonsqueezing}
    There exists a symplectic embedding $\iota:B_{n,1}(1)\hookrightarrow B_{n,1}(\alpha,\infty)$ if and only if $\alpha \geq 1$.
\end{maintheorem}

\begin{figure}[ht]
  \centering
  \begin{tikzpicture}
    \begin{scope}[shift={(-7,-2)}]
        \fill[opacity=0.1]
            (0,0) -- (0,2) -- (6,4);
        \draw[thick] 
            (0,2) -- (0,0) -- (6,4);
        \node at 
            (3,1) {$B_{n,1}(1)$};
        \draw[dashed] 
            (0,2) -- (6,4);
        \draw[dashed] 
            (1,1) node[cross]{} -- (0,0);
    \end{scope}
    
    \draw[right hook->] 
        (-1,0) -- node[anchor=south] {$s$} (1,0);
    \draw[thick, red]
        (-0.6,-0.6) -- (0.6,0.6);
    \draw[thick, red]
        (0.6,-0.6) -- (-0.6,0.6);  
    \node[text=red] at 
        (0,-1) {No!};
    \begin{scope}[shift={(2,-2)}]
        \fill[opacity=0.1]
            (0,0) -- (0,4) -- (3,4) -- (3,2);
        \draw[thick] 
            (0,4) -- (0,0) -- (3,2);
        \node at 
            (3,1) {$B_{n,1}(\alpha,1)$};
        \draw[dashed] 
            (3,4) -- (3,2);
        \draw[dashed] 
            (1,1) node[cross]{} -- (0,0);
    \end{scope}

  \end{tikzpicture}
  \caption{For $\alpha <1$ the "quantitative" rational homology ball $B_{n,1}(1)$ does not embed into the rational homology cylinder of "width" $\alpha$.}
  \label{fig:non_squeezing_intro}
\end{figure}
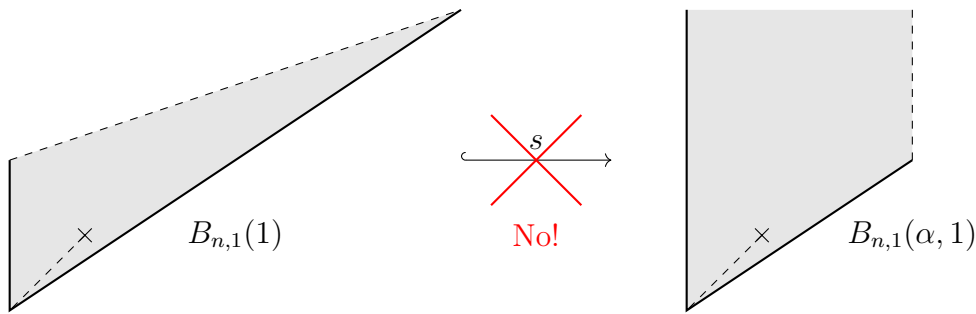

We will introduce the necessary notation in \cref{subsec:compactification}, and the proof of this theorem is given in \cref{sec:corollaries}.

\begin{rem}
    As with Gromov's non-squeezing theorem this theorem opens up a wide range of embedding questions.
    For example one can start looking for symplectic staircases for the rational homology ellipsoid and prove two-ball theorems.
    The techniques developed and discussed in this paper also give a way to tackle these problems by simple computations. 
    The first steps in this direction are proven in the forthcoming paper \cite{BrEvHaSch25}.
\end{rem}

The first result that we prove in this paper concerns the special case $n=2$ of Theorem \ref{thrm:B}, namely determining which symplectic forms on $X_1$ carry a Lagrangian $\mathbb{R}P^2$.
We give special attention to this case and use it as an introduction for two reasons. 
The first reason is pedagogical; even though it is the simplest case, it gives a very good overview of the more complicated ideas and arguments used in the rest of the paper. 
Secondly, it is related to the following question that Kronheimer asks in \cite{Kro11:question}.

\begin{question}\emph{(\cite[Kronheimer]{Kro11:question})}\label{q:Kro}
    Does every simply-connected, non-spin symplectic $4$-manifold contain a Lagrangian $\mathbb{R}P^2$?
\end{question}

We will answer this question in the negative:

\begin{maintheorem}[\cref{thrm: kronh}]\label{thrm:Kro_squeeze_rp2}
The symplectic manifold $(X_1,\omega_{h,\mu})$, by which we denote the one-point blow-up of size $\mu$ of the complex projective plane with the Fubini-Study form that integrates to $h$ over a line, carries a Lagrangian $\mathbb{R}P^2$ if and only if 
\[\mu<\frac{h}{2}.\]
In particular, for any symplectic form with $\mu\geq\frac{h}{2}$, the symplectic non-spin $4$-manifold $(X_1,\omega_{\mu,h})$ does \textbf{not} contain an embedded Lagrangian $\mathbb{R}P^2$.
\end{maintheorem}

Proving this theorem in \cref{subsec:Kro_q} will serve as an introduction to the rest of this paper in the sense that the main techniques used in this paper are contained within the proof of \cref{thrm:Kro_squeeze_rp2}.

\subsection{Organization}
The paper is organized as follows: In \cref{sec: Prelim} we set notation and recall the relevant properties of almost toric fibrations and Lagrangian pinwheels, along with a brief discussion about their self-intersection properties and their first Chern class.
In \cref{sec:ratblow} we review useful properties of rational and ruled symplectic $4$-manifolds, as well as how they behave under the rational blow-up. 
This section also contains a short proof that rationally blowing-up a positive symplectic manifold yields a positive symplectic manifold. 
We then demonstrate the above ideas to answer Kronheimer's question in \cref{subsec:Kro_q}. 
In \cref{Section-1-Construction} we use the theory of almost toric fibrations to construct Lagrangian pinwheels in rationally ruled surfaces. 
In addition we describe a useful compactification of the $B_{n,1}$ that we will use for proving the non-squeezing theorem. 
In \cref{sec:obstruction} we determine the manifold obtained by blowing up a liminal pinwheel in a rationally ruled surface.
This is the main input for the rest of the applications, immediately proving \cref{thrm:A,thrm:B}. 
Finally, \cref{sec:corollaries} contains the proofs of \cref{thrm:Kho,thrm:nonsqueezing}.

\subsection{Acknowledgments}

The authors would like to thank Jo\'{e} Brendel, Jonny Evans, Felix Schlenk, Joel Schmitz, Federica Pasquotto and George Politopoulos for many insightful and supportive discussions. 
In particular, we would like to thank Jonny Evans for inspiring most of the content of this paper as well as suggesting proofs of \cref{prop: discrepancies} and \cref{Lag_pinwheels_self_int} and Felix Schlenk for guiding us to the non-squeezing result proven in \cref{sec:corollaries}.
Finally, we would like to express our gratitude to our common MSc Thesis supervisor, Will Merry. 
This paper is surely part of his mathematical legacy and his outstanding teaching.

\section{Preliminaries}\label{sec: Prelim}

\subsection{Notation}

Denote the standard symplectic form on $S^2$ by $\omega_0$, where by standard we mean that $\omega_0$ gives $S^2$ area $1$.
For positive real numbers $a$ and $b$ define the split symplectic form $\omega_{a,b}\vcentcolon= a\omega_0 \oplus b\omega_0$ on $S^2\times S^2$.
Furthermore, define the two homology classes $A\vcentcolon= [S^2 \times \{pt\}]$ and $B\vcentcolon= [\{pt\} \times S^2]$, that form a basis of $H_2(S^2 \times S^2; \mathbb{Z})$.

The complex projective space equipped with the Fubini--Study form scaled such that the area of a complex line is equal to $h >0$ is denoted by $(\mathbb{C}P^2,\omega_h)$. 
The one-point blowup of $(\mathbb{C}P^2,\omega_h)$ of size $0<\mu<h$ will be denoted by $(X_1,\omega_{h,\mu})$. 
In general, the $k$-fold blowup of $(\mathbb{C}P^2,\omega_h)$ will be denoted by $(X_k,\omega_{h,\bm{\mu}})$, where $\bm{\mu}=(\mu_1,\ldots,\mu_k) \in \mathbb{R}_{>0}^k$ is the vector of weights of the blowups. 
The class of a complex line will be denoted by $H = [\mathbb{C}P^1] \in H_2(\mathbb{C}P^2;\mathbb{Z})$ and the classes defined by the exceptional divisors will be denoted by $E_i \in H_2(X_k;\mathbb{Z})$.

\subsection{Recollections of almost toric fibrations}
\label{subsec:atfs}

The constructive part of the proof of \cref{thrm:A} and \cref{thrm:B} is heavily based on the theory of almost toric fibrations, in the following abbreviated by ATFs, which was originally introduced in \cite{Sym02:Fourtwo}. 
ATFs proved to be a very powerful tool in four-dimensional symplectic geometry, especially for providing constructions. 
For example, Vianna used them in \cite{Via16} and \cite{Via17} to construct examples of exotic Lagrangian tori in $\mathbb{C}P^2$ and in monotone del Pezzo surfaces. 
Vianna's papers use sequences of mutations to construct exotic Lagrangian tori. 
In these sequences of mutations Lagrangian pinwheels naturally appear and the numerics of these pinwheels exhibit a stunning connection to the Markov equation.\footnote{For an introduction to these sequences of mutations compare \cite[Appendix I]{Ev23:Book}.} 
These sequences of mutations hint at surprising connections between algebraic geometry and symplectic geometry, which was worked out by Evans and Smith \cite{EvSm18}.

In our discussion of ATFs we will rely on the setup, the notation and the conventions in \cite{Ev23:Book} and \cite{Ev24:KIAS}.
Without going into the details, we recall crucial aspects of the theory and recall some calculations that play an important role in our construction. 
An ATF is given by a singular Lagrangian torus fibration on a four dimensional symplectic manifold, whose singularities are elliptic-regular, elliptic-elliptic or focus-focus type.
There are three essential operations in this context called \textit{nodal trade}, \textit{nodal slide} and \textit{mutation}. For more details on these operations compare \cite[Sections 8.1,8.2 and 8.3]{Ev23:Book}. 
We will quickly introduce the latter two in order to fix some notation. 

The most important operation in our construction are so called \textit{mutations} of almost toric base diagrams. 
Assume that $\Delta \subseteq \mathbb{R}^2$ is an almost toric base diagram decorated with the positions of the base-nodes together with branch cuts emanating from each base-node in the direction of the eigenline of the monodromy associated to the respective base-node, connecting the base-node to the toric boundary. 
Suppose that $x\in \Delta$ is a base node with associated eigenline pointing in the $(p,q)$-direction, where $(p,q)$ is a primitive integral eigenvector of the monodromy. 
In this case the clockwise monodromy is given by the matrix 
\begin{equation*}
    M(p,q)\vcentcolon=
    \left(
        \begin{matrix}
        1-pq & -q^2\\
        p^2 & 1+pq
        \end{matrix}
    \right).
\end{equation*} 

The line $\langle(p,q)\rangle$ bisects the almost toric base diagram into two pieces $\Delta_1$ and $\Delta_2$, where $\Delta_1$ lies clockwise of the branch cut. 
Rotating the branch cut by $180$ degrees in the anticlockwise direction has the effect of replacing $\Delta_2$ by $\Delta_2 M(p,q)^{-1}$ to obtain the new almost toric base diagram $\Delta_1 \cup \Delta_2 M(p,q)^{-1}$. 
This manipulation, that neither effects the underlying symplectic manifold nor the ATF defined on it, is called a \textit{mutation}.\footnote{Note that Evans uses a slightly different language in the sense that in \cite{Ev23:Book} a mutation consist of a nodal slide combined with a "change of branch cut".}
In \cref{fig:mutation} an example of a mutation is illustrated, which will appear again and again in our construction.\footnote{Evans and Urz\'ua introduce exactly this configuration in \cite{EvaUrz21} for the purpose of finding explicit embeddings of Lagrangian pinwheels in certain algebraic surfaces. This configuration is only a specific case of the much more general configuration they are considering.}

The other manipulation of almost toric base diagrams that we will use are \textit{nodal slides}. 
A nodal slide is a family of ATFs in which a base-node is moved along the direction defined by the branch cut that is emanating from the base-node. 
Recall that sliding a base-node along the branch cut line can be arranged in such a way that the ATF is only changed in an arbitrarily small, contractible neighborhood of the segment along which the base-node moved. 
Note that this does not change the underlying symplectic manifold. However, it changes the almost toric fibration.

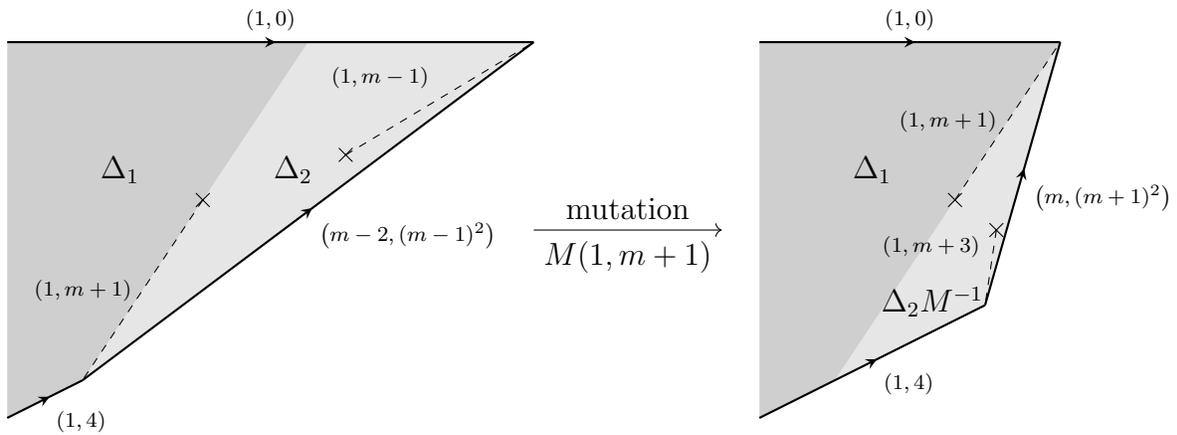
\begin{figure}[ht]
  \centering
  \begin{tikzpicture}
    \begin{scope}[shift={(-3.5,-2.5)}]
        \fill[opacity=0.1] 
            (0,5) -- (0,0) -- (1,0.5) -- (7,5);
        \fill[opacity=0.1] 
            (0,5) -- (0,0) -- (1,0.5) -- (4,5);
        \draw[thick, mid arrow] 
            (0,0) -- node[anchor=north west] {\tiny $(1,4)$} 
            (1,0.5);
        \draw[thick, mid arrow]
            (1,0.5) -- node[anchor=north west] {\tiny $\big(m-2,(m-1)^2\big)$} 
            (7,5);
        \draw[thick, mid arrow]
            (0,5) -- node[anchor=south] {\tiny $(1,0)$}
            (7,5);
        \node 
            at (1.5,3.3) {$\Delta_1$};
        \node 
            at (3.8,3.3) {$\Delta_2$};
        \draw[dashed] 
            (2.6,2.9) node[cross] {} -- node[anchor=east] {\tiny $(1,m+1)$}(1,0.5);
        \draw[dashed] 
            (4.5,3.5) node[cross] {} -- node[anchor=south east] {\tiny $(1,m-1)$}(7,5);
    \end{scope}
    
    \draw[->] (3.5,0) -- node[anchor=south] {mutation} node[anchor=north] {$M(1,m+1)$} (6,0);
    
    \begin{scope}[shift={(6.5,-2.5)}]
        \fill[opacity=0.1] 
            (0,5) -- (0,0) -- (3,1.5) -- (4,5);
        \fill[opacity=0.1] 
            (0,5) -- (0,0) -- (1,0.5) -- (4,5);
        \draw[thick, mid arrow] 
            (0,0) -- node[anchor=north west] {\tiny $(1,4)$} 
            (3,1.5);
        \draw[thick, mid arrow]
            (3,1.5) -- node[anchor=north west] {\tiny $\big(m,(m+1)^2\big)$} 
            (4,5);
        \draw[thick, mid arrow]
            (0,5) -- node[anchor=south] {\tiny $(1,0)$}
            (4,5);
        \node at (1.5,3.3) {$\Delta_1$};
        \node at (2.3,1.55) {$\Delta_2M^{-1}$};
        \draw[dashed] 
            (2.6,2.9) node[cross] {} -- node[anchor=east] {\tiny $(1,m+1)$}(4,5);
        \draw[dashed] 
            (3.15,2.5) node[cross] {} -- node[anchor=south east] {\tiny $(1,m+3)$}(3,1.5);
    \end{scope}

  \end{tikzpicture}
  \caption{For $m \in \mathbb{N}_{>1}$ two almost toric base diagrams related by a mutation at the lower base-node of almost toric base diagram on the left. 
  A calculation, that is easily reproduced, yields all the directions involved on the right hand side.}
  \label{fig:mutation}
\end{figure}

\subsection{Lagrangian Pinwheels}\label{Subsection_Lagrangian_Pinwheels}

Let us recall some notions and definitions around Lagrangian pinwheels, following \cite{BrSch24}. Compare also the discussions in \cite{Kho13:Sympratblo} and \cite{EvSm18}. 

\begin{definition}
    A $p$-pinwheel, for $p\in \mathbb{N}_{\geq 2}$, is the topological space obtained by gluing a $2$-cell to a circle along the map $z\rightarrow z^{p}$.
\end{definition}

Pinwheels where first studied by Khodorovskiy in \cite{Kho13:Sympratblo}, where pinwheels were studied as the Lagrangian skeleta of certain Stein manifolds.
These are the rational homology balls $B_{p,q}$, for relatively prime integers $p$ and $q$, carrying a standard symplectic form $\omega_{std}$. 
The contact boundary of these rational homology balls is the Lens space $L(p^2,pq-1)$. 

Another point of view that is related to this paper is Evans' exposition in \cite{Ev24:KIAS}. 
There, the rational homology balls $B_{p,q}$ appear as Milnor fibres of the cyclic quotient singularity $\frac{1}{p^2}(1,pq-1)$. 
From this point of view, the link of this singularity is the lens space ${L(p^2,pq-1)}$ and the Lagrangian $(p,q)$-pinwheel appears as the vanishing cycle. 
In these notes Evans also explains how viewing the rational homology balls as smoothings of cyclic quotient singularities translates into almost toric geometry. 
The almost toric base diagram, which we will denote by $\Delta_{p,q}$, that corresponds to the rational homology ball $B_{p,q}$ is depicted in \cref{fig:atfrationalhomologyball}. 
The associated link, i.e.\ the lens space, lives over a horizontal line away from the branch cut, and the vanishing cycle, i.e.\ the Lagrangian $(p,q)$-pinwheel, is a visible Lagrangian above the branch cut.

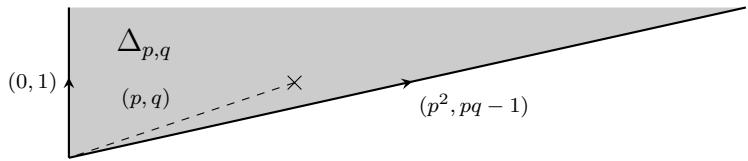
\begin{figure}[ht]
  \centering
  \begin{tikzpicture}
    \begin{scope}
        \fill[opacity=0.2] 
            (0,2) -- (0,0) -- (9,2);
        \draw[thick, mid arrow] 
            (0,0) -- node[left] {\tiny $(0,1)$} 
            (0,2);
        \draw[thick, mid arrow]
            (0,0) -- node[anchor=north west] {\tiny $(p^2,pq-1)$} 
            (9,2);
        \node at (1,1.5) {$\Delta_{p,q}$};
        \draw[dashed] 
            (3,1) node[cross]{} -- node[anchor=south east] {\tiny $(p,q)$} (0,0);
    \end{scope}
  \end{tikzpicture}
  \caption{The almost toric base diagram $\Delta_{p,q}$.}
  \label{fig:atfrationalhomologyball}
\end{figure}

The first concrete and general definition of Lagrangian $L_{p,q}$-pinwheels appeared in \cite[Definition 2.3]{EvSm18}, as certain Lagrangian immersions of discs, extending the one given in \cite[Definition 3.1]{Kho13:Sympratblo} for $L_{n,1}$-pinwheels.
Khodorovskiy also proved in \cite[Section 3.3]{Kho13:Sympratblo} that Lagrangian $L_{n,1}$-pinwheels admit standard symplectic neighborhoods, in the same way that usual Lagrangian submanifolds do.
In fact, her proofs extend directly to the more general Lagrangian $L_{p,q}$-pinwheels.
We will instead follow Brendel and Schlenk \cite[Section 1]{BrSch24} and use the following equivalent definition:

\begin{definition}\label{def:Lag_pinwheels}
    In a symplectic four-manifold $(X,\omega)$, a Lagrangian $(p,q)$-pinwheel, denoted by $L_{p,q}$, is a $p$-pinwheel that admits an open neighborhood symplectomorphic to an open neighborhood of the Lagrangian skeleton of the rational homology ball $B_{p,q}$.
\end{definition}

\begin{remark}\label{remark_pinwheel=qhb}
    We will sometimes blur the distinction between embeddings of Lagrangian $L_{p,q}$-pinwheels and of the associated rational homology balls since one implies the other and we will often only care about existence of the embeddings, as for example in \cref{thrm:Kho}. 
    On the other hand, the existence of Lagrangian $L_{p,q}$-pinwheels does not give any quantitative information about the embeddings of the rational homology balls, for example how big or small the $B_{p,q}$ may be.

    Moreover, in the following by an $L_{p,q}$-pinwheel, we will always mean a Lagrangian $L_{p,q}$-pinwheel.
\end{remark}

\begin{example}
    For $(p,q)=(2,1)$ the corresponding $L_{2,1}$-pinwheel is just a Lagrangian~$\mathbb{R}P^2$, and $B_{2,1}$ is symplectomorphic to the unit cotangent disc bundle of $\mathbb{R} P^2$. 
Note that in this case we are talking about honest embeddings. 
\end{example}

We give the pinwheels that we are concerned with a special name.

\begin{definition}\label{def:liminal_pinwheel}
A \textbf{liminal pinwheel} is a $L_{n,1}$-pinwheel which satisfies the following properties:
\begin{enumerate}
    \item 
        If $n=2k$, then $L_{2k,1}$ lives in $X_1$ and represents the $\mathbb{Z}_{2k}$-homology class $kH+(k+1)E$. 
        Furthermore, $L_{2k,1}$ is disjoint from a smooth sphere either in the $\mathbb{Z}$-class $(k+1)H-kE$ or in the $\mathbb{Z}$-class $2H$.
    \item 
        If $n=2k+1$, then $L_{2k+1,1}$ lives in $S^2\times S^2$ and represents the $\mathbb{Z}_{2k+1}$-homology class $A+kB$. 
        Furthermore, $L_{2k+1,1}$ is disjoint from a smooth sphere either in the $\mathbb{Z}$ class $A+(k+1)B$ or in the $\mathbb{Z}$ class $2A+B$.
\end{enumerate}
\end{definition}

\begin{remark}
    The adjective \textit{liminal} refers to the fact that these pinwheels essentially behave like the visible pinwheels we construct in \cref{Section-1-Construction}. 
    So in a sense they are "almost visible" pinwheels.
    The introduction and proofs in \cref{sec:obstruction} will make the relevance of the definition clear.

    However, it is unclear to us, if every $L_{n,1}$-pinwheel, with homological data as in \cref{def:liminal_pinwheel} is liminal.\footnote{It is not hard to see, which means that is follows from a lengthy but easy computation, that, for example, a Lagrangian $\mathbb{R}P^2$ in $X_1$ has to intersect all embedded spheres representing the classes $(k+1)H-kE$ for $k\geq2$, which are in the orthogonal complement of the liminal $L_{2k,1}$-pinwheels. This follows by using ideas similar to the ones used in the proof of \cite[Proposition 1.4]{BoLiWu13}, relying on a topological result of Kervaire and Milnor \cite{KeMi61}.}
\end{remark}

The definition of Lagrangian pinwheels we use, i.e. \cref{def:Lag_pinwheels}, is directly connected to the Weinstein-type theorem that Khodorovskiy proves in \cite[Lemma 3.4.]{Kho13:Sympratblo}.\footnote{Note that the definition we use is is really tailored towards talking about embedding of pairs $(B_{p,q},L_{p,q})$. There are slightly more general definition of Lagrangian pinwheel embeddings. For more details on this compare \cite[Section 2]{EvSm18}. However, for the purpose of this paper these definitions are equivalent in the sense that our theorems hold no matter which definition is used.}
This suggests that there should be a topological obstruction to find pinwheel embeddings in symplectic manifolds by ways of the identification of the tangent bundle with the normal bundle on the smooth part of the pinwheel embeddings, similar to the usual equation $\chi(L)=-[L]\cdot[L]$ for a closed oriented Lagrangian submanifold $L$ of a symplectic $4$-manifold. 
The following lemma gives the analogue to this equation.

\begin{lemma}\label{Lag_pinwheels_self_int}
    Suppose that $(X^4,\omega)$ contains an $L_{p,q}$-pinwheel. Denote by $[L] \in H_2(X;\mathbb{Z}_p)$ its homology class. Then,
    \begin{align*}
        [L]\cdot[L] \equiv -1 \pmod{p}.
    \end{align*}
    \begin{proof}
        By definition the $L_{p,q}$-pinwheel admits an open neighborhood in $X$ that is symplectomorphic to an open neighborhood of the Lagrangian skeleton of the rational homology ball $B_{p,q}$. 
        Since the self-intersection can be computed locally we may compute the self-intersection number within the rational homology ball.
        Consider the almost toric base diagram $\Delta_{p,q}$ shown in \cref{fig:atfrationalhomologyball_self_int}. 
        The $L_{p,q}$-pinwheel is a visible Lagrangian living over the branch cut and therefore projects to the red segment shown in the almost toric base diagram. 
        
        Now we define a configuration defining the same homology class as $L_{p,q}$. 
        Consider a configuration of line segments as illustrated in \cref{fig:atfrationalhomologyball_self_int} in orange.
        One of the two pieces is a vertical segment and the other piece is parallel to the branch cut, that points in the $(p,q)$-direction. 
        The segment terminating at the node is the projection of a disk, that is obtained by transporting the vanishing cycle along the line segment. 
        This means that intersecting a smooth fibre over the segment with the disk defines the $(-p,q)$-class. Note that this is not a visible Lagrangian, but still a visible surface in the sense of Symington \cite[Section 7]{Sym02:Fourtwo}. 
        The other segment parallel to the branch cut defines a Lagrangian $(p,q)$-pinwheel core.
        \begin{figure}[ht]
            \centering
            \begin{tikzpicture}
                \begin{scope}
                \fill[opacity=0.2] 
                    (0,2) -- (0,0) -- (9,2);
                \draw[thick, mid arrow] 
                    (0,0) -- node[anchor=east] {\tiny $(0,1)$} (0,2);
                \draw[thick, mid arrow]
                    (0,0) -- node[anchor=north west] {\tiny $(p^2,pq-1)$} 
                    (9,2);
                \draw[thick, red] 
                    (0,0) node[circlecross] {} -- (3,1);
                \draw[thick, orange] 
                    (0,0.5) node[circlecross] {} -- (3,1.5) -- (3,1);
                \draw[dashed] 
                    (3,1) node[cross] {} -- node[anchor=east] {\tiny $(p,q)$} (0,0);
                \end{scope}
            \end{tikzpicture}
            \caption{The configuration used to compute the self-intersection number of the visible pinwheel.}
        \label{fig:atfrationalhomologyball_self_int}
        \end{figure}
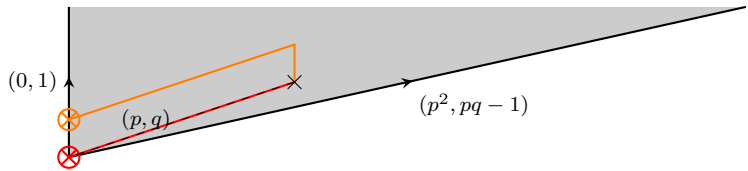
        Choosing the Lagrangian pinwheel core in an appropriate way the pinwheel core and the disk match over the intersection point of the two line segments. 
        This defines a $(p,q)$-pinwheel that is no longer Lagrangian but defines the same homology class as $L_{p,q}$, which follows from the fact that this configuration is homotopic to $L_{p,q}$. 
        This configuration and the visible $L_{p,q}$-pinwheel intersect only in the nodal fibre and the intersection is transverse. 
        Symington shows in \cite[Lemma 7.11]{Sym02:Fourtwo}, by using the local model around the node, that the sign associated to this intersection is $-1$.\footnote{Another way to immediately see this is by Picard--Lefschetz theory, because the above proof shows that computing the self-intersection number of these visible Lagrangian pinwheels is essentially the same as computing the self-intersection number of vanishing thimbles in Picard--Lefschetz theory, the difference being a sign change because of the difference between having a Lagrangian fibration and a holomorphic fibration.}
        This implies that 
        \begin{align*}
            [L]\cdot[L] \equiv -1 \pmod{p}.
        \end{align*}
    \end{proof}
\end{lemma}

\begin{remark}\label{remark:pinwheels_soft_obstruction}
    Note that \cref{Lag_pinwheels_self_int} when combined with in \cite[Lemma 2.13]{EvSm18} gives that if $L$ is an $L_{p,q}$-pinwheel in a symplectic $4$-manifold $(X,\omega)$ it needs to satisfy
    \[[L]^2=-1 \mod p \quad\text{and}\quad c_1(X)([L])=q \mod p.\]
    These constraints often provide useful soft obstructions to symplectic embeddings of rational homology balls. 
    For example, they can be used to show that many, and possibly all, of the smooth embeddings constructed by Owens in \cite{Ow20:nonsymp} cannot be made symplectic, without appealing to the deep theorem about Markov numbers of Evans--Smith \cite{EvSm18}. 
    Applied in a different context, they also show that any $L_{p,q}$-pinwheel in $\mathbb{C} P^2$ needs to satisfy $q^2=-9 \mod p$ without any appeal to number theoretic techniques, as was done in \cite{EvSm18}.
    We plan to give more interesting and precise applications of these obstructions in future work. 
\end{remark}

Surprisingly, Lemma \ref{Lag_pinwheels_self_int} is the key input to determine when a homology class has a disjoint representative from a an $L_{p,q}$ pinwheel. 
See \cref{Appendix:orthogonal complement} for more details.

\section{Rational blow-ups on rational and ruled surfaces}
\label{sec:ratblow}

\subsection{Rational and ruled symplectic surfaces}

Rational symplectic surfaces are smooth symplectic $4$-manifolds $(X,\omega)$ where $X$ is smoothly diffeomorphic to $S^2\times S^2$ or $X_k=\mathbb{C} P^2\#k \overline{\mathbb{C} P^2}$. 
In \cite{Gr85} Gromov initiated the study of the symplectic topology of rational symplectic surfaces, studying $\mathbb{C} P^2$ and the monotone $(S^2\times S^2,\omega_0 \oplus \omega_0)$. 
Since then, his results have been vastly generalized in various directions, mainly through the developments of Taubes--Seiberg--Witten theory. 
For general introductions to rational and ruled surfaces in symplectic topology compare \cite{Wen18:rationalbook,FrMcD96:classrat,FrMcD96:classrat_book,McDSal16:Intro_book, McD:rat_rul}.

\begin{theorem}[For example \cite{McDSal16:Intro_book}]\label{thrm:symp_forms_on_rational}
    Let $\omega_1,\omega_2$ be two symplectic forms on a rational manifold $X$. 
    If $[\omega_1]=[\omega_2]$ as elements of $H_{dR}^2(X;\mathbb{R})$, then $(X,\omega_1)$ and $(X,\omega_2)$ are diffeomorphic. The same holds for ruled symplectic surfaces.
\end{theorem}

This makes symplectic rational manifolds easy to work with, since $(X,\omega)$ depends only on the smooth topology of $X$ and the cohomology class $[\omega] \in H^2_{dR}(X;\mathbb{R})$.

Rational and ruled symplectic surfaces, or symplectic Hirzebruch surfaces, are the symplectic rational surfaces $(X,\omega)$ where $X$ is diffeomorphic to $S^2\times S^2$ or $X_1$, or, equivalent, the rational surfaces with second Betti number $b_2=2$. 
A direct application of Theorem \ref{thrm:symp_forms_on_rational} tells us that there are model symplectic forms on rational and ruled surfaces.

\begin{lemma}
Let $(X^4,\omega)$ be a rational and ruled symplectic manifold.
\begin{enumerate}
    \item 
        If $X=S^2\times S^2$ then $\omega$ diffeomorphic to $\omega_{a,b}$, the product symplectic form giving the product factors area $a>0$ and $b>0$ respectively.
    \item 
        If $X=X_1$ then $\omega$ is diffeomorphic to $\omega_{h,\mu}$, the form giving the line and the exceptional divisor areas $h$ and $\mu$ respectively.\footnote{Note that there is the volume constraint $h>\mu$.}
\end{enumerate}
\end{lemma}

\begin{remark}
Another way to rephrase the above lemma is that any symplectic form on $S^2\times S^2$ or $X_1$ makes the ambient symplectic manifold have the structure of a symplectic $S^2$-fibration over $S^2$. 
The "ruled" adjective comes from the $S^2$-fibration property, whereas the "rational" adjective comes from the fact that $S^2$ is the base space.
\end{remark} 

Therefore statements on the symplectic topology of rational and ruled symplectic surfaces can be reduced to the study of $(S^2\times S^2,\omega_{a,b})$ and $(X_1,\omega_{h,\mu})$.

Rational and Ruled symplectic surfaces can be constructed by considering a \textit{standard symplectic disc bundle} $(E,\tau)$ over $(S^2,\omega_0)$ and then compactifying the total space of the unit disc bundle of $E$ via symplectic reduction.
This is explained thoroughly in \cite{Bi01:Lagbar}, where the notion of a standard symplectic disc bundle was introduced.\footnote{A toric point of view of this construction can be found in \cite{Ev23:Book}.} 
By Weinstein's neighborhood theorem, standard symplectic disc bundles are models for the neighborhood of symplectic submanifolds and provide convenient ways to describe them. 
When $E\rightarrow S$ has Euler class $n$, then the total space of this line bundle is exactly $V_{n}$ in Khodorovskiy's terminology in \cite{Kho14:Embratball}.

\begin{lemma}\label{lemma:compactification}
Let $(V_{-n-1},\tau)$ be a neighborhood of an $(-n-1)$-symplectic sphere $S$, where the normal disc has area $f$ and the zero-section has area $s$. 
Let $(X,\omega)$ be the rational and ruled symplectic manifold obtained by compactifying via a symplectic cut. 
Then
\begin{enumerate}
    \item 
        if $n=2k$ then $(X,\omega)$ is symplectomorphic to $(X_1,\omega_{h,\mu})$ with $f=h-\mu$ and $s=-kh+(k+1)\mu$. 
        In particular $\mu>\frac{k}{k+1}h$.
    \item 
        if $n=2k+1$ then $(X,\omega)$ is symplectomorphic to $(S^2\times S^2,\omega_{a,b})$. 
        Assuming that $[S]=A-(k+1)B\in H_2(S^2\times S^2,\mathbb{Z})$ we have $b=f$, $a=s+(k+1)f$ and, in particular, $\frac{a}{k+1}>b$. 
\end{enumerate}
\end{lemma}

\begin{proof}
By construction the manifold $(X,\omega)$ admits a symplectic $S^2$-fibration over $S^2$ that has two distinguished symplectic submanifolds; the zero-section $S$ and the fibers of the fibration.
Assume that $F$ is such a fibre and that the areas are $\omega(S)=s$ and $\omega(F)=f$, respectively.

If $S\cdot S=-2k-2$ then the $S^2$-fibration admits a section with even self intersection and therefore $X$ is $S^2\times S^2$. 
Up to diffeomorphism, i.e.\ swapping $A$ and $B$, we can assume that $[F]=B$ and $[S]=A-(k+1)B$. 
Therefore $\omega$ is diffeomorphic to $\omega_{a,b}$ where $b=f$ and $a=s+(k+1)f$ and thus, since $s>0$ by virtue of $S$ being symplectic, we have also that $a>(k+1)b$.
Similarly, if $S\cdot S=-2k-1$ then $X$ is $X_1$. 
The only possibility is that $[F]=H-E$ and $[S]=-kH+(k+1)E$. 
Thus $(X,\omega)$ is symplectomorphic to $(X_1,\omega_{h,\mu})$ where $\mu=kf+s$ and $h=(k+1)f+s$. 
Since $s=-kh+(k+1)\mu$ and $s>0$, we also get that $kh<(k+1)\mu$.
\end{proof}

\begin{remark}
    Note that the above lemma can also be proven by ways of the toric interpretation as described in \cite{Ev23:Book}.
\end{remark}

\subsection{Blowing-up Lagrangian \texorpdfstring{$L_{n,1}$}{Ln,1}-pinwheels}\label{Section_Blowing_Up_Pinwheels}

As mentioned before, we will use the symplectic rational blow-up to show that the existence of certain $L_{n,1}$-pinwheels implies necessarily the inequalities for the symplectic form on $S^2\times S^2$ or $X_1$ claimed in \cref{thrm:A} and \cref{thrm:B}.

First we will recall the basics on how to perform the symplectic rational blow-up and respectively the blow-down. 
Then, we will determine the effect of the rational blow-up along an $L_{n,1}$-pinwheel, namely compute in which homology classes the spheres introduced via the blow-up live. 
Finally, we will leverage this information to infer the claimed inequalities.

The symplectic rational blow-up is a symplectic surgery introduced by Khodorovskiy in \cite{Kho13:Sympratblo} and is the inverse operation of a more classical procedure, the symplectic rational blow-down, introduced by Symington in \cite{Sym98:Ratblo}. 
For a streamlined exposition of these operations we refer to \cite[Chapter 9]{Ev23:Book}. 
Even though both operations are defined for general $L_{p,q}$-pinwheels, we will focus only on the case of $L_{n,1}$-pinwheels, which are considered in this paper.

The main idea of the rational blow-up comes from the observation that the rational homology ball $B_{n,1}$ has the same contact boundary as the Weinstein neighborhood of a certain configuration of symplectic spheres $C_n$, namely the lens space $L(n^2,n-1)$.
The configuration $C_n$ consists of a collection of $n-1$ symplectic spheres $S_i$ for $0\leq i\leq n-2$, with prescribed intersection pattern.
For $1\leq i\leq n-3$ the sphere $S_i$ intersects only $S_{i-1}$ and $S_{i+1}$ in exactly exactly one point respectively and this intersection is transverse.
Moreover, the self-intersection numbers of the spheres are given by $S_0^2=-(n+2)$ and $S_i^2=-2$ for $1\leq i\leq n-2$.\footnote{This observation is best explained through algebraic geometry. 
As mentioned in \cref{Subsection_Lagrangian_Pinwheels} the lens space $L(n^2,n-1)$ is the link of the cyclic quotient singularity $\frac{1}{n^2}(1,n-1)$. 
$B_{n,1}$ is the Milnor Fiber of this singularity, while $C_n$ is the resolving divisor of a minimal resolution of the singularity. 
This is explained in detail in \cite[Chapter 9]{Ev23:Book} along the lines of the original reference \cite{Sym98:Ratblo}}

\begin{definition}[Symplectic rational blow-up]
    Let $(X^4,\omega)$ be a $4$-dimensional symplectic manifold and suppose that $L_{n,1}$ is a Lagrangian pinwheel in $(X^4,\omega)$. 
    The symplectic manifold $(\widetilde{X},\widetilde{\omega})$, obtained by removing a small standard neighborhood of $L_{n,1}$ and gluing back a neighborhood of a specific configuration of symplectic spheres $C_n$, is called the symplectic $\textbf{rational blow-up}$ of $(X,\omega)$ along $L_{n,1}$.
\end{definition}

\begin{remark}
    By the nature of the symplectic rational blow-up, there is a symplectic identification of $\nu L_{n,1}$, the neighborhood of the Lagrangian pinwheel, and $\nu C_n$, the neighborhood of the symplectic spheres. 
    In other words, we may identify symplectically
    \[\big(X-\nu L_n,\omega\big)\simeq \big(\widetilde{X}-\nu C_n,\widetilde{\omega}\big).\]
\end{remark}

Expectedly, given a $C_n$ configuration in some $(\widetilde{X},\widetilde{\omega})$, one may perform the \textbf{symplectic rational blow-down} to replace it with an $L_{n,1}$-pinwheel.
Determining how surgeries affect manifolds can be a very subtle problem. 
In our case however, using heavy machinery of Seiberg-Witten theory we can, a priori, determine the resulting symplectic manifold.

\begin{theorem}[\cite{PaShi23}]
    Let $(X^4,\omega)$ be a rational symplectic manifold and let $L$ be an $L_{n,1}$-pinwheel therein. 
    Then, the symplectic manifold $(\widetilde{X},\widetilde{\omega})$ obtained by performing a rational blow-up along $L$ is again a rational symplectic manifold.
\end{theorem}

We give an alternative proof of this theorem here. 
The main idea of the proof is straight-forward: The rational blow-up is localized around a pinwheel, which is \textbf{null-homologous} in rational homology and de Rham cohomology. 
Therefore, the cohomology classes of the new Chern class and symplectic form are, roughly, the sum of the old ones and the discrepancies introduced by the chain of spheres $C_n$. 
Therefore, the result we want to prove is reduced to a local computation for the chain of spheres. 
This mode of proof, suggested to us by J. Evans, is directly inspired by \cite{AdaEv24}, where it was shown that a very different kind of symplectic surgery, Luttinger surgery along a Lagrangian Klein bottle, also preserves rationality.

\begin{proposition}\label{prop: discrepancies}
    Let $(X,\omega)$ be a simply connected symplectic $4$-manifold. 
    Suppose that $(X,\omega)$ carries an $L_{p,q}$-pinwheel. 
    Then, the symplectic rational blow-up along $L_{p,q}$ gives rise to a new symplectic manifold $(\widetilde{X},\tilde{\omega})$ for which
    \[c_1(\widetilde{X})\cdot\tilde{\omega}=c_1(X)\cdot \omega + d_{p,q}.\]
    The number $d_{p,q}$ is always positive and depends only on the pinwheel $L_{p,q}$ and the symplectic area of the divisors introduced by the blow up.  
\end{proposition}

\begin{proof}
    Let $U$ be the standard neighborhood of $L_{p,q}$, diffeomorphic to $B_{p,q}$.
    Since $B_{p,q}$ is a rational homology ball and its contact boundary is the lens space $L(p^2;pq-1)$ the inclusion $X-U\xhookrightarrow{i} X$ induces the isomorphism on $H_2(X-U,\mathbb{Q})\xhookrightarrow{i_*}H_2(X,\mathbb{Q})$. 
    Therefore we can find cycles $Y_i$ disjoint from $L\subset U$ that form a basis of $H_2(X,\mathbb{Q})$. Since the rational blow-up happens around $L$, these cycles correspond to cycles $\tilde{Y}_i\subset \widetilde{X}$.

    Since the lens space $L(p^2;pq-1)$ related to $B_{p,q}$ is a rational homology sphere, Mayer--Vietoris shows that the spheres $S_j$, introduced by the rational blow-up, together with the cycles $\widetilde{Y}_i$, form a basis of $H_2(\widetilde{X},\mathbb{Q})$.
    Therefore, after identifying rational homology with de Rham cohomology via Poincaré duality, we have the following expressions for the Poincaré duals of $c_1(X)$ and $\omega$:
    \[c_1(X)=\sum \delta_i Y_i\quad\text{and}\quad \omega=\sum \beta_i Y_i,\]
    Then, the corresponding classes for $\widetilde{X}$ become:
    \[c_1(\widetilde{X})=\sum d_jS_j+\sum \delta_i\widetilde{Y}_i\quad\text{and}\quad \tilde{\omega}=\sum b_jS_j +\sum \beta_i\widetilde{Y}_i.\]
    Since $S_j\cdot \widetilde{Y}_i=0$ we easily see that
    \[c_1(\widetilde{X})\cdot\tilde{\omega}=c_1(X)\cdot \omega+\overbrace{(\sum d_jS_j)(\sum b_rS_r)}^{d_{p,q}}\]
    and so we only need to show that $d_{p,q}>0$. This follows directly since
    \[d_{p,q}=\sum_j d_j(\sum_r b_rS_jS_r)=\sum_j d_ja_j\]
    where $a_j=\tilde{\omega}(S_j)$, namely the symplectic areas of the $S_j$. The $a_i$ are the areas of symplectic spheres so they are strictly positive. 
    In addition, it is a classical fact that the discrepancies of cyclic quotient singularities are non-negative\footnote{Algebraic geometers usually measure the discrepancies in terms of the canonical class of the complex structure, while we consider them in terms of the Chern class. This creates the obvious sign difference between the $d_j$ here and the $c_j$ of \cite{HaTeUr17,EvSm20:Bounds}.}, with at least one of them  positive, compare \cite{HaTeUr17,EvSm20:Bounds}. 
    Therefore $d_{p,q}$ is always positive. 
\end{proof}

\begin{corollary}\label{cor: rationality preserv}
    Let $(X,\omega)$ be a symplectic rational manifold with $c_1(X)\cdot\omega\geq 0$. 
    Then, rationally blowing up an $L_{p,q}$-pinwheel in $X$ will produce a symplectic rational manifold.
\end{corollary}

\begin{proof}
    This follows immediately from \cref{prop: discrepancies} and the Liu--Ohta--Ono criterion \cite{Liu96, OhtaOno96} that a simply connected symplectic manifold with $c_1(\widetilde{X})\cdot\tilde{\omega}>0$ is rational. 
\end{proof}

A straightforward Mayer--Vietoris argument shows that, when blowing up along an $L_{n,1}$-pinwheel in a rational manifold, the second Betti number $b_2$ increases by the number of spheres in $C_n$, namely $n-1$, i.e.\ we obtain $b_2(\widetilde{X})=b_2(X)+ (n-1)$. 
In particular, this gives us the following lemma.

\begin{lemma}\label{Lemma_Rational_Blowup_in_S2S2_X1}
    Performing a symplectic rational blow-up along an $L_{n,1}$-pinwheel in either $(S^2\times S^2,\omega_{a,b})$ or $(X_1,\omega_{\mu,h})$ produces the manifold $X_{n}$ equipped with a symplectic form $\omega_{h,\bm{\mu}}$ for some $h >0$ and weights $\bm{\mu}\in \mathbb{R}_{>0}^n$.
\end{lemma}

\subsection{A motivating example: Lagrangian \texorpdfstring{$\mathbb{R}P^2$}{RP2}s and Kronheimer's question}
\label{subsec:Kro_q}

Before we delve into the proofs of the main theorems of the paper, we think it is worthwhile to showcase the ideas behind the proofs contained in this paper by answering, in the negative, the following question posed by Kronheimer.
\begin{question}\emph{(\cite[Kronheimer]{Kro11:question})}
    Does every simply-connected, non-spin symplectic $4$-manifold contain a Lagrangian $\mathbb{R}P^2$?
\end{question}
Answering this question uses the ideas of proof for the simplest pinwheel, namely $L_{2,1}$-pinwheels, i.e.\ Lagrangian $\mathbb{R}P^2$s. 
Most of the ideas that we have to implement to study general $L_{n,1}$-pinwheels are already needed in the $n=2$ case. 
In addition, this case usually needs slightly different treatment than the other cases and so we will deal with it here and free ourselves in order to make the convenient assumption that $n\geq 3$ for the rest of the paper. 
We emphasize that some ideas leading to this method of proof are contained in \cite{BoLiWu13}, were first realized in \cite{SmSh20} for Lagrangian $\mathbb{R}P^2$s and later in \cite{Ada24:Pin} for $L_{3,1}$-pinwheels and intersection patterns of Lagrangian $\mathbb{R}P^2$s. 
The answer to Kronheimer's question is given by the following theorem.

\begin{theorem}\label{thrm: kronh}
    The symplectic manifold $(X_1,\omega_{h,\mu})$ carries a Lagrangian $\mathbb{R}P^2$ if and only if 
    \[\mu<\frac{h}{2}.\]
    In particular, for any symplectic form with $\mu\geq\frac{h}{2}$, the symplectic non-spin 4-manifold $(X_1,\omega_{\mu,h})$ does \textbf{not} contain any Lagrangian $\mathbb{R}P^2$.
\end{theorem}

The idea of proof is that if there existed some Lagrangian $\mathbb{R}P^2$ in $(X_1,\omega_{h,\mu})$, where $\mu\geq\frac{h}{2}$, blowing it up would create a smooth manifold equipped with a symplectic form that cannot exist.
To show this we need to figure out the symplectic manifold that is obtained by performing a rational blow-up along a Lagrangian $\mathbb{R}P^2$ in $(X_1,\omega_{h,\mu})$.\footnote{We chose the language of rational blow-up, because it fits into the larger picture of this paper. The process of blowing up Lagrangian $\mathbb{R}P^2$s was known before the more general rational blow-up was introduced. For example in \cite[Subsection 2.2]{BoLiWu13} it is called the "symplectic cutting construction" and like the usual blow-up in the symplectic category it is realized by symplectically cutting along the periodic geodesic flow on the boundary of a Weinstein neighborhood.} 
Because the blow-up only effects a neighborhood of the $\mathbb{R}P^2$ that retracts onto the $\mathbb{R}P^2$, we first want to understand what homology classes live in its complement.

\begin{lemma}\label{le:disjoint_cycles_RP2}
    The only homology class of $H_2(X_1;\mathbb{Z}_2)$ that can carry an embedded Lagrangian $\mathbb{R}P^2$ is $H$. 
    In addition, for any smooth $\mathbb{R}P^2$ representing the class $H$ we can find embedded oriented submanifolds disjoint from it, representing the classes $2H$ and $2H-E$ in $H_2(X_1;\mathbb{Z})$.
\end{lemma}

\begin{proof}
    Let $L$ be an embedded Lagrangian $\mathbb{R}P^2$. 
    A direct application of Audin's criterion \cite{Aud88}, that implies that $[L]^2\equiv 1 \mod 4$, means that $[L]=H$, because no other $\mathbb{Z}_2$-homology class satisfy the condition on the Pontryagin square.
    
    The rest of the lemma is a straightforward. 
    The $\mathbb{Z}$-homology classes $2H$ and $E$ both pair trivially$\mod 2$ with the class $H$ and thus \cite[Lemma 4.10]{BoLiWu13} and \cite[Section 2]{SmSh20} shows that $2H$ and $2H-E$ can be represented by embedded oriented submanifolds, disjoint from $L$.\footnote{The homological complement to Lagrangian pinwheels is discussed in more detail in \cref{Appendix:orthogonal complement}.}
\end{proof}

Now, we can determine the smooth manifold obtained by rationally blowing up a Lagrangian $\mathbb{R}P^2$ in $(X_1,\omega_{h,\mu})$, which also means determining the cohomology class that is coming with, i.e.\ determining the periods of the symplectic form that it carries.\footnote{Notice that by Theorem \ref{thrm:symp_forms_on_rational} the periods indeed determine the symplectic form.}

\begin{lemma}\label{lemma: blowing up rp2}
    Let $L$ be an embedded Lagrangian $\mathbb{R}P^2$ in $(X_1,\omega_{h,\mu})$. Rationally blowing up $L$ yields the symplectic manifold $(X_2,\omega_{\tilde{h},\bm{\tilde{\mu}}})$, where
    \begin{equation*}
    \tilde{h}=\frac{3h}{2}-\mu+\frac{c}{4},\quad \tilde{\mu_1}=h -\mu+\frac{c}{2}\quad\text{and}\quad\tilde{\mu}_2=\frac{h}{2}-\mu-\frac{c}{4}.
    \end{equation*}
    Here $c$ is the symplectic area of the $(-4)$-sphere introduced by the rational blow up.
\end{lemma}

\begin{proof}
    By Lemma \ref{Lemma_Rational_Blowup_in_S2S2_X1}, performing the rational blow-up along $L$ indeed gives us the symplectic manifold $(X_2,\omega_{\tilde{h},\bm{\tilde{\mu}}})$. 
    Our goal is to now determine the periods $\tilde{h},\tilde{\mu_1}$ and $\tilde{\mu}_2$.
    
    Let $S_0\subset X_2$ be the symplectic $(-4)$-sphere obtained by the rational blow-up of $L$. 
    The lattice generated by the homology classes 
    \begin{equation*}
        \langle W_1=2H-E,W_2=2H \rangle \subseteq H_2(X_1;\mathbb{Z})
    \end{equation*}
    embeds into $H_2(X_2;\mathbb{Z})$, disjoint from $S_0$.\footnote{Again, this will be discussed in more detail for the general case of $L_{p,q}$-pinwheels in \cref{Appendix:orthogonal complement}. However, for the case needed here, i.e. of a Lagrangian $\mathbb{R}P^2$, this is discussed in detail in \cite[Lemma 4.10]{BoLiWu13} and \cite[Section 2]{SmSh20}.}
    Let $D_1$ and $D_2$ be the classes corresponding to $W_1$ and $W_2$ under the blow-up. 
    Denote by $\{\widetilde{H},\widetilde{E}_1,\widetilde{E}_2\}$ the canonical basis of $H_2(X_2,\mathbb{Z})$. 
    A direct calculation involving self-intersections, intersections, and Chern numbers shows that, up to a diffeomorphism, we can assume that $D_1=2\widetilde{H}-\widetilde{E}_1$  and $D_2=3\widetilde{H}-2\widetilde{E}_1-\widetilde{E}_2$. 
    The fact that these classes are disjoint from $S_0$ together with the adjunction formula imply that $S_0=-\widetilde{H}+2\widetilde{E}_1-\widetilde{E}_2$.\footnote{We could have also used Proposition \ref{prop: determining d_i} to determine the $D_i$, as we will do for the higher order pinwheels. The point is that for $n\geq 9$ blowing up an $L_{n,1}$-pinwheel in $\mathbb{C}P^2$ yields $X_{n\geq 9}$ and there the analogous homological information is not enough to determine the classes of $S_i$ and $D_i$ completely.}
    From the above correspondence we get that
    \begin{align*}
    2h-\mu=\omega_{h,\mu}(W_1)&=\omega_{\tilde{h},\tilde{\bm{\mu}}}(D_1)=2\tilde{h}-\tilde{\mu}_1\\
    2h=\omega_{h,\mu}(W_2)&=\omega_{\tilde{h},\tilde{\bm{\mu}}}(D_2)=3\tilde{h}-2\tilde{\mu}_1-\tilde{\mu}_2.
    \end{align*}
    The symplectic sphere $S_0$ has positive area $c>0$. Therefore, we also have 
    \begin{equation*}
        c=\omega_{\tilde{h},\tilde{\bm{\mu}}}(S_0)=-\tilde{h}+2\tilde{\mu}_1-\tilde{\mu}_2.
    \end{equation*}
    Finally, solving for the periods of $\omega_{\tilde{h},\tilde{\bm{\mu}}}$ we get
    \begin{equation*}
        \tilde{h}=\frac{3h}{2}-\mu+\frac{c}{4},\quad \tilde{\mu_1}=h -\mu+\frac{c}{2}\quad\text{and}\quad\tilde{\mu}_2=\frac{h}{2}-\mu-\frac{c}{4}.
    \end{equation*}
\end{proof}

\begin{remark}
    Note that there is an abundance of information to deduce the weights of the symplectic form in the above lemma. 
    For example, one could also use the information coming from the total volume of the underlying $(X_1,\omega_{h,\mu})$ and compare it to the total volume of $(X_2,\omega_{\tilde{h},\bm{\tilde{\mu}}})$.
\end{remark}

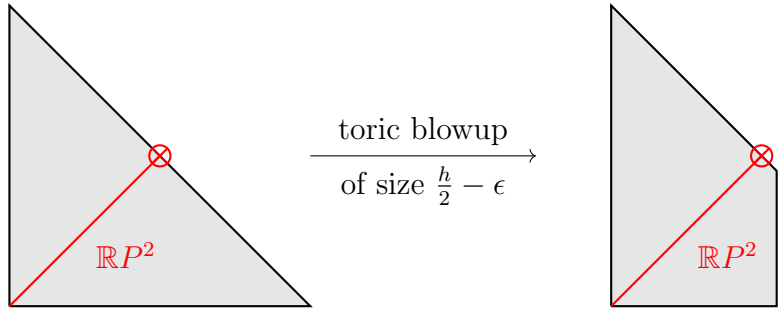
\begin{figure}[ht]
  \centering
  \begin{tikzpicture}
  
    \begin{scope}[shift={(-5,-2)}]
        \fill[opacity=0.1] (0,4) -- (0,0) -- (4,0);
        \draw[thick] (0,0) -- (4,0) -- (0,4) -- (0,0);
        \draw[thick,red] (2,2) node[circlecross] {} -- node[below right]{$\mathbb{R}P^2$} (0,0);
    \end{scope}
    
    \draw[->] (-1,0) -- node[anchor=south] {toric blowup} node[anchor=north] {of size $\frac{h}{2}-\epsilon$} (2,0);
    
    \begin{scope}[shift={(3,-2)}]
        \fill[opacity=0.1] (0,4) -- (0,0) -- (2.2,0) -- (2.2,1.8);
        \draw[thick] (0,0) -- (2.2,0) -- (2.2,1.8) -- (0,4) -- (0,0);
        \draw[thick,red] (2,2) node[circlecross] {} -- node[below right]{$\mathbb{R}P^2$} (0,0);
    \end{scope}

  \end{tikzpicture}
  \caption{Construction of a Lagrangian $\mathbb{R}P^2$-pinwheel in $(X_1,\omega_{h,\mu})$ for $\mu<\frac{h}{2}$.}
  \label{fig:RP2_X_1}
\end{figure}

We can now conclude the proof of \cref{thrm:Kro_squeeze_rp2}.

\begin{proof}[Proof of \cref{thrm:Kro_squeeze_rp2}]
    Suppose that there exists a Lagrangian $\mathbb{R}P^2$ in $(X_1,\omega_{h,\mu})$. \cref{lemma: blowing up rp2} shows that blowing up this Lagrangian $\mathbb{R} P^2$ yields the symplectic manifold $(X_2,\omega_{\tilde{h},\bm{\tilde{\mu}}})$, where
    \[\tilde{\mu}_2=\frac{h}{2}-\mu-\frac{c}{4}\]
    and $c$ corresponds to the size of the rational blow-up. By \cite[Theorem A]{LiLiu95:Genadj} we know that since $\widetilde{E}_2$ is represented by a smooth sphere it must also be represented by a symplectic one, therefore $\tilde{\mu}_2$ is positive, which directly implies 
    
    \begin{equation*}
        \mu<\frac{h}{2}.
    \end{equation*}
    
    Now, assuming $\mu<\frac{h}{2}$ there are various ways to construct Lagrangian $\mathbb{R} P^2$s in $(X_1,\omega_{h,\mu})$. In \cref{fig:RP2_X_1} we illustrate a toric approach. This is further elaborated on in \cref{Section-1-Construction}.
    Alternatively, one can use Biran's Lagrangian barrier result in \cite{Bi01:Lagbar} that shows that the standard $\mathbb{R}P^2$ in $\mathbb{C}P^2$ is a Lagrangian barrier. 
    In our situation, this means that in $(\mathbb{C}P^2,\omega_h)$ one can fit a symplectic ball of capacity \textbf{at most} $\mu<\frac{h}{2}$ disjoint from the standard Lagrangian $\mathbb{R}P^2$ and so blowing up that ball gives the desired symplectic form on $X_1$ that is carrying a Lagrangian $\mathbb{R}P^2$.
\end{proof}

\begin{remark}
    Note that \cref{thrm: kronh} actually reproves Biran's Lagrangian barrier result in \cite[Theorem 1.B]{Bi01:Lagbar} for the $\mathbb{C}P^2$ case.
\end{remark}

\section{Constructing Lagrangian pinwheels}
\label{Section-1-Construction}

\subsection{Lagrangian \texorpdfstring{$L_{2k+1,1}$}{L2k+1,1}-Pinwheels in \texorpdfstring{$S^2\times S^2$}{S2xS2}}\label{Section_Pinwheels_S2xS2}

Since the question out of which this paper grew was posed by Evans in \cite[Appendix J]{Ev23:Book}, and subsequently answered by the first author in \cite{Ada24:Pin}, let us quickly recall the construction that Evans used to illustrate his question.

Assume that $a<b<2a$ and perform the sequence of manipulations of the standard toric fibration of $(S^2 \times S^2,\omega_{a,b})$ illustrated in \cref{fig:Evanspinwheel_a_b_2a}.
Then we obtain a visible Lagrangian $L_{3,1}$-pinwheel, living over the red segment in the almost toric base diagram on the right of \cref{fig:Evanspinwheel_a_b_2a}.
This is because over the red segment (excluding the point on the toric boundary) there lives a Lagrangian disk with center on the singularity of the nodal fibre, and over the red segment (excluding the base-node) there lives a $(3,1)$-pinwheel core.

Note that the assumption on the weights of the symplectic form $\omega_{a,b}$ is essential, because for $2a \leq b$ the visible Lagrangian either hits the vertex, in which case the visible Lagrangian is a sphere with a Schoen--Wolfson singular point, as discussed in \cite[Section 5.2]{Ev23:Book}, or hits the toric boundary on the right, in which case the visible Lagrangian is a smooth disk.

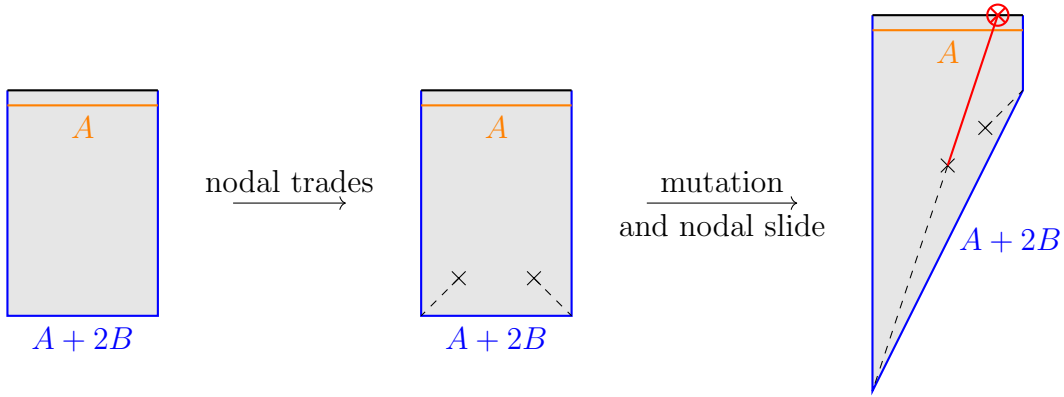
\begin{figure}[ht]
  \centering
  \begin{tikzpicture}
  
    \begin{scope}[shift={(-3,-1.5)}]
      \fill[opacity=0.1] (0,3) rectangle (2,0);
      \draw[thick, blue] (0,3) -- (0,0) -- node[anchor=north] {$A+2B$} (2,0) -- (2,3);
      \draw[thick] (0,3) -- (2,3);
      \draw[thick, orange] (0,2.8) -- node[anchor=north] {$A$}(2,2.8);
    \end{scope}
    
    \draw[->] (0,0) -- node[anchor=south] {nodal trades} node[anchor=north] {} (1.5,0);
    
    \begin{scope}[shift={(2.5,-1.5)}]
      \fill[opacity=0.1] (0,3) rectangle (2,0);
      \draw[thick, blue] (0,3) -- (0,0) -- node[anchor=north] {$A+2B$} (2,0) -- (2,3);
      \draw[thick, orange] (0,2.8) -- node[anchor=north] {$A$}(2,2.8);
      \draw[thick] (0,3) -- (2,3);
      \draw[dashed] (0.5,0.5) node[cross] {} -- (0,0);
      \draw[dashed] (1.5,0.5) node[cross] {} -- (2,0);
    \end{scope}

    \draw[->] (5.5,0) -- node[anchor=south] {mutation} node[anchor=north] {and nodal slide} (7.5,0);

    \begin{scope}[shift={(8.5,-0.5)}]
      \fill[opacity=0.1] (0,3) -- (0,-2) -- (2,2) -- (2,3) -- (0,3);
      \draw[thick, blue] (0,3) -- (0,-2) -- node[anchor=west] {$A+2B$} (2,2) -- (2,3);
      \draw[thick, orange] (0,2.8) -- node[anchor=north] {$A$}(2,2.8);
      \draw[thick] (0,3) -- (2,3);
      \draw[dashed] (1.5,1.5) node[cross] {} -- (2,2);
      \draw[dashed] (1,1) node[cross] {} -- (0,-2);
      \draw[thick,red] (5/3,3) node[circlecross] {} -- (1,1);
    \end{scope}
    
  \end{tikzpicture}
  \caption{Evans' construction of a $L_{3,1}$-pinwheel in $(S^2 \times S^2,\omega_{a,b})$.}
  \label{fig:Evanspinwheel_a_b_2a}
\end{figure}

If $b \leq a$ the sequence of manipulations of the standard toric fibration gives a slightly different picture, but essentially works out the same way.
Now we want to determine the $\mathbb{Z}_3$-homology class that is represented by this Lagrangian pinwheel. 
Denote this homology class by $L\in H_2(S^2 \times S^2; \mathbb{Z}_3)$. 
The visible pinwheel is disjoint from the sphere living over the orange part of the toric boundary, which represents the class $A+2B \in H_2(S^2 \times S^2; \mathbb{Z}_3)$, as in indicated in \cref{fig:Evanspinwheel_a_b_2a}.\footnote{By definition the collection of the three symplectic sphere, indicated in blue in the honest toric fibration of $(S^2 \times S^2,\omega_{a,b})$ depicted in \cref{fig:Evanspinwheel_a_b_2a}, represents the class $A+2B$. To see that this class is also represented by the symplectic sphere living over the blue part of the toric boundary in the almost toric fibration depicted in \cref{fig:Evanspinwheel_a_b_2a} one can investigate the local model of a nodal trade as mentioned in \cite[Remark 8.6]{Ev23:Book}, or just prove a small topological lemma relying on the fact the the homology of a ball is trivial.}
Moreover, over the orange line lives a standard symplectic representative of the class $A$. 
This implies that $A\cdot L \equiv 1 \mod 3$, because we exactly know how $L$ looks like in the fibers of the almost toric fibration. 
This is discussed in \cite[Section 7]{Sym02:Fourtwo}.
Using these two observations one concludes that $L=A+B \in H_2(S^2 \times S^2; \mathbb{Z}_3)$. 

Assuming that $b<a<2b$ and realizing that the sequence of manipulations of standard toric fibration of $(S^2 \times S^2,\omega_{a,b})$ can also be performed in the other ``horizontal'' direction we conclude:

\begin{corollary}
    There exists a Lagrangian $L_{3,1}$-pinwheel in $(S^2\times S^2,\omega_{a,b})$ representing the $\mathbb{Z}_3$-homology class $A+B \in H_2(S^2 \times S^2; \mathbb{Z}_3)$ if
        \[\frac{a}{2}<b<2a.\]
\end{corollary}

\begin{remark}
    Note that using the sequence of mutations illustrated in \cref{fig:Evanspinwheel_a_b_2a} in the other, i.e.\ ``horizontal'', direction one ends up with a visible $L_{3,1}$-pinwheel whose pinwheel core terminates on the toric boundary that represents the class $B$. We do not know wether these two pinwheels are isotopic in any sense.
\end{remark}

The main observation of this section is that one can continue mutating at the lower base-node in \cref{fig:Evanspinwheel_a_b_2a} to produce $L_{2k+1,1}$-pinwheels in $(S^2 \times S^2,\omega_{a,b})$ for $k\geq 1$. We distinguish three cases, but assume $\frac{a}{2}<b$ throughout in the following.

For the first case assume that $a<b$ and that there exists an integer $l\in \mathbb{N}$ such that $la<b<(l+1)a$. 
The sequence of mutations illustrated in \cref{fig:pinwheel_la_b_(l+1)a} yields an $L_{2l+1,1}$-pinwheel, as well as an $L_{2(l+1)+1,1}$-pinwheel.
Note that the $L_{2l+1,1}$-pinwheel is disjoint from an embedded symplectic sphere representing the class $A+2B \in H_2(S^2 \times S^2; \mathbb{Z}_{2l+1})$ as well as from an embedded symplectic sphere representing the class $(l+1)A+B\in H_2(S^2 \times S^2; \mathbb{Z}_{2l+1})$\footnote{Note that the vertex in the middle almost toric fibration in \cref{fig:pinwheel_la_b_(l+1)a} that is not associated to the pinwheel is Delzant. Retracting the branch cut gives the desired sphere.} and that the intersection between the $\mathbb{Z}_{2l+1}$-class $L$ represented by the pinwheel and the class $A$ can be immediately deduced from the diagram\footnote{As before compare \cite[Section 7]{Sym02:Fourtwo} for details on the intersection properties, as well as for the behavior of the homology classes under almost toric manipulations of the almost toric base diagram. This is also discussed in \cite[Chapter 8]{Ev23:Book}.}:
\[ L\cdot (A+2B)=0 \mod{2l+1} \hspace{1cm}\text{ and }\hspace{1cm} L\cdot A=1 \mod{2l+1}, \]
which implies that $L=lA+B$.

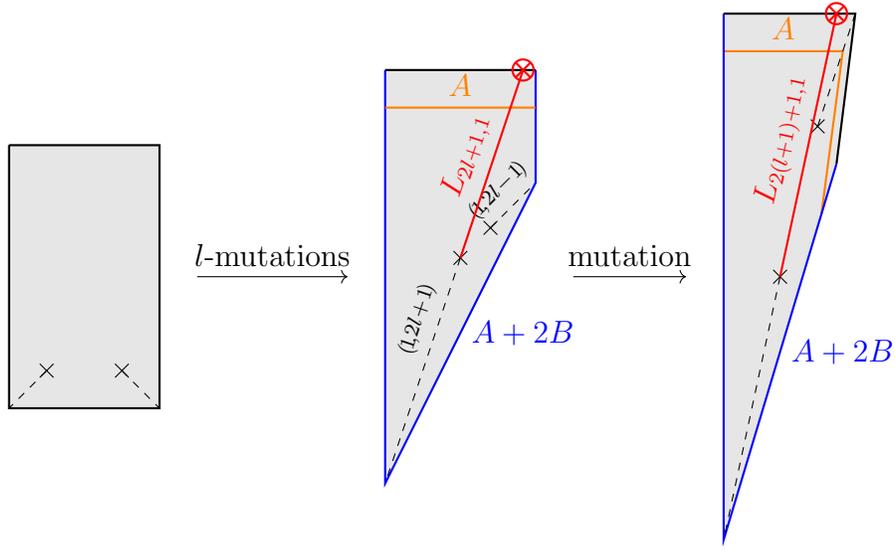
\begin{figure}[ht]
  \centering
  \begin{tikzpicture}
  
    \begin{scope}[shift={(-4,-1.75)}]
        \fill[opacity=0.1] (0,3.5) -- (0,0) -- (2,0) -- (2,3.5);
        \draw[thick] (0,3.5) -- (0,0) -- (2,0) -- (2,3.5) -- (0,3.5);
        \draw[dashed] (0.5,0.5) node[cross] {} -- (0,0);
        \draw[dashed] (1.5,0.5) node[cross] {} -- (2,0);
    \end{scope}
    
    \draw[->] (-1.5,0) -- node[anchor=south] {$l$-mutations} node[anchor=north] {} (0.5,0);
    
    \begin{scope}[shift={(1,-2.75)}]
        \fill[opacity=0.1] (0,5.5) -- (0,0) -- (2,4) -- (2,5.5);
        \draw[thick, blue] (0,5.5) -- (0,0) -- node[right]{$A+2B$} (2,4) -- (2,5.5);
        \draw[thick] (2,5.5) -- (0,5.5);
        \draw[thick, orange] (0,5) -- node[anchor=south]{$A$} (2,5);
        \draw[dashed] (1,3) node[cross] {} -- node[above right,sloped]{\tiny $(\!1\!,\!2l\!+\!1\!)$} (0,0);
        \draw[dashed] (1.4,3.4) node[cross] {} -- node[above,sloped]{\tiny $(\!1\!,\!2l\!-\!1\!)$}(2,4);
        \draw[thick,red] (5.5/3,5.5) node[circlecross] {} -- node[above, sloped]{$L_{2l+1,1}$} (1,3);
    \end{scope}

    \draw[->] (3.5,0) -- node[anchor=south] {mutation} node[anchor=north] {} (5,0);
    
    \begin{scope}[shift={(5.5,-3.5)}]
        \fill[opacity=0.1] (0,7) -- (0,0) -- (3/2,5) -- (7/4,7) ;
        \draw[thick] (0,7) -- (7/4,7) -- (3/2,5);
        \draw[thick, orange] (0,6.5) -- node[anchor=south]{$A$} (19/12,13/2) -- (26/20,26/6);
        \draw[thick, blue] (0,7) -- (0,0) -- node[right]{$A+2B$} (3/2,5);
        \draw[dashed] (3/4,7/2) node[cross] {} -- node[anchor=east] {} (0,0);
        \draw[dashed] (5/4,11/2) node[cross] {} -- (7/4,7);
        \draw[thick,red] (6/4,14/2) node[circlecross] {} -- node[above, sloped]{$L_{2(l+1)+1,1}$} (3/4,7/2);
    \end{scope}
    
  \end{tikzpicture}
  \caption{Construction of an $L_{2l+1,1}$-pinwheel and an $L_{2(l+1)+1,1}$-pinwheel in $(S^2 \times S^2,\omega_{a,b})$ for all $l\leq k$. (This is a schematic picture, the lengths and slopes are not shown to proportion.)}
  \label{fig:pinwheel_la_b_(l+1)a}
\end{figure}

Mutating at the lower base-node of the almost toric base diagram on the right of \cref{fig:pinwheel_la_b_(l+1)a} once more yields the almost toric base diagram shown schematically in \cref{fig:pinwheel_la_b_(l+1)_final}. 
Now notice that the part on the right of this diagram is the same as the one illustrated on the left in \cref{fig:mutation} for $m=2(l+2)$. 
Since the mutation in \cref{fig:mutation} can be iterated we conclude that for all integers $l\leq k$ there exists an $L_{2k+1,1}$-pinwheel in $(S^2 \times S^2,\omega_{a,b})$ representing the homology class $kA+B \in H_2(S^2 \times S^2; \mathbb{Z}_{2k+1})$.

The second case is the case in which $b$ is an integer multiple of $a$, in which case slightly different almost toric base diagrams appear, but the almost toric base diagrams appearing are essentially the same.

The last case is the case in which $b<a$, which yields a sequence of mutations that is slightly different, but essentially shows the same phenomenon.

Combining the discussion of these three cases we obtain:

\begin{corollary}\label{Corollary_Construction_S2xS2}
    There exists an $L_{2k+1,1}$-pinwheel in $(S^2 \times S^2,\omega_{a,b})$ representing the class $kA+B \in H_2(S^2 \times S^2; \mathbb{Z}_{2k+1})$ if
        \[\frac{a}{2}<b<(k+1)a.\]
    Furthermore, we can assume that these pinwheels are disjoint from an embedded symplectic sphere representing the class $A+2B\in H_2(S^2\times S^2;\mathbb{Z})$.
    
    Moreover, assuming that for $k\in\mathbb{N}$ we have $ka<b<(k+1)a$, the $L_{2k+1,1}$-pinwheel can be assumed to be disjoint from an embedded symplectic sphere representing the class $(k+1)A+B\in H_2(S^2\times S^2;\mathbb{Z})$.
    
    Equivalently, an $L_{2k+1,1}$-pinwheel in $(S^2 \times S^2,\omega_{a,b})$ representing the class $A+kB \in H_2(S^2 \times S^2; \mathbb{Z}_{2k+1})$ exists if
        \[\frac{b}{2}<a<(k+1)b\]
    and the properties about disjoint symplectic spheres hold analogously.
\end{corollary}

\begin{remark}
    The second part of \cref{Corollary_Construction_S2xS2} follows by symmetry. 
    In the calculations of the obstruction in \cref{sec:obstruction} we will consider Lagrangian pinwheels that represent the class $A+kB$, because this is what we used for the calculations initially.
    Since the construction part was inspired by Evans' question we followed his approach in our construction which yield pinwheels in the class $kA+B$.
\end{remark}

\begin{figure}[ht]
  \centering
  \begin{tikzpicture}
    
    \begin{scope}
        \fill[opacity=0.1] 
            (0,8) -- (0,0) -- (2,2) -- (5,8) ;
        \draw[thick] 
            (0,8) -- (5,8);
        \draw[thick, mid arrow]
            (2,2) -- node[right]{\tiny $\big(2(l+1),(2(l\!+\!1)\!+\!1)^2\big)$} (5,8);
        \draw[thick, mid arrow]
            (0,0) -- node[right]{\tiny $(1,4)$} (2,2);
        \draw[thick, orange] 
            (0,7) -- node[anchor=north]{$A$} (4,7) -- (7/3,10/3) -- (0,1);
        \draw[thick, blue] 
            (0,8) -- node[left]{$A+2B$} (0,0);
        \draw[dashed] 
            (3,6) node[cross] {} -- node[above,sloped]{\tiny $(1,2(l\!+\!2)\!+\!1)$} (2,2);
        \draw[dashed] 
            (3.5,6.5) node[cross] {} -- node[above,sloped]{\tiny $(1,2(l\!+\!1)\!+\!1)$} (5,8);
        \draw[thick,red] 
            (3.5,8) node[circlecross] {} -- node[above, sloped]{$L_{2(l+2)+1,1}$} (3,6);
    \end{scope}
    
  \end{tikzpicture}
  \caption{Schematic picture of the almost toric fibration on $(S^2 \times S^2,\omega_{a,b})$ after mutating at the lower base-node in the almost toric base diagram in \cref{fig:pinwheel_la_b_(l+1)a}. In this almost toric base diagram an $L_{2(l+2)+1,1}$-pinwheel appears as a visible Lagrangian.}
  \label{fig:pinwheel_la_b_(l+1)_final}
\end{figure}
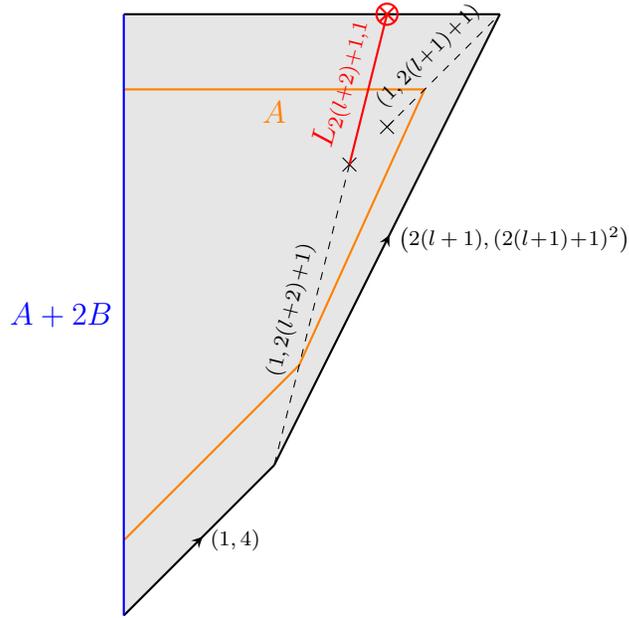

\subsection{Lagrangian \texorpdfstring{$L_{2k,1}$}{L2k,1}-Pinwheels in \texorpdfstring{$X_1$}{X1}}

The procedure to obtain $L_{2k,1}$-pinwheels in $X_1$ is analogous to the construction in the previous subsection. 
We illustrate the sequence of manipulations of the toric fibration of $(X_1,\omega_{h,\mu})$ in \cref{fig:pinwheels_X_1}. 
After a further mutation at the upper base-node in the right hand almost toric base diagram in \cref{fig:pinwheels_X_1} the sequence of manipulations continues in strict analogy to the sequence of mutations in \cref{Section_Pinwheels_S2xS2}.\footnote{The fact that there is no symmetry involved eliminates one of the cases in \cref{Section_Pinwheels_S2xS2}.} 

A calculation then yields that given $0<\mu<h$ there exists an $L_{2k,1}$-pinwheel in $(X_1,\omega_{h,\mu})$ if $(k+1)\mu-k h<0$.\footnote{The considerations are exactly the same as in \cref{Section_Pinwheels_S2xS2} and involve no new ideas. Therefore they are omitted here.}
Moreover, denoting the class represented by the visible $L_{2k,1}$-pinwheel by $L \in H_2(X_1;\mathbb{Z}_{2k})$, we find the two conditions 
\[ L\cdot 2H=0 \mod{2k} \hspace{1cm}\text{ and }\hspace{1cm} L\cdot (H-E)=1 \mod{2k}, \]
which means that $L=kH+(k+1)E \in H_2(X_1;\mathbb{Z}_{2k})$. Therefore, we obtain:

\begin{corollary}\label{Corollary_Construction_X1}
    There exists a Lagrangian $L_{2k,1}$-pinwheel in $(X_1,\omega_{h,\mu})$ representing the class $kH+(k+1)E \in H_2(X_1; \mathbb{Z}_{2k})$ if
        \[\mu<\frac{k}{k+1}h.\]
    Furthermore, we can assume that these pinwheels are disjoint from an embedded symplectic sphere representing the class $2H\in H_2(X_1;\mathbb{Z})$. 
    
    Moreover, assuming that for $k\in \mathbb{N}$ we have 
        \[\frac{k-1}{k}h<\mu<\frac{k}{k+1}h,\]
    the $L_{2k,1}$-pinwheel can be assumed to be disjoint from an embedded symplectic sphere representing the class $(k+1)H-kE\in H_2(X_1;\mathbb{Z})$.
\end{corollary}

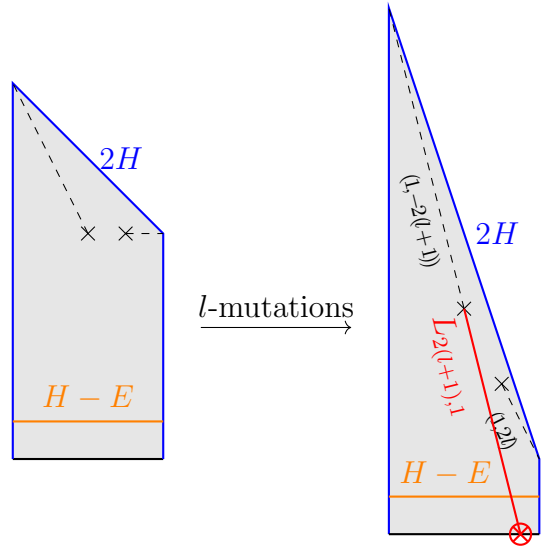
\begin{figure}[ht]
  \centering
  \begin{tikzpicture}
  
    \begin{scope}[shift={(-4,-1.75)}]
        \fill[opacity=0.1] (0,5) -- (0,0) -- (2,0) -- (2,3);
        \draw[thick] (0,0) -- (2,0);
        \draw[thick, blue] (0,0) -- (0,5) -- node[right]{$2H$} (2,3) -- (2,0);
        \draw[thick, orange] (0,0.5) -- node[above]{$H-E$} (2,0.5);
        \draw[dashed] (1,3) node[cross] {} -- (0,5);
        \draw[dashed] (1.5,3) node[cross] {} -- (2,3);
    \end{scope}
    
    \draw[->] (-1.5,0) -- node[anchor=south] {$l$-mutations} node[anchor=north] {} (0.5,0);
    
    \begin{scope}[shift={(1,-2.75)}]
        \fill[opacity=0.1] (0,7) -- (0,0) -- (2,0) -- (2,1);
        \draw[thick, blue] (0,0) -- (0,7) -- node[right]{$2H$} (2,1) -- (2,0);
        \draw[thick] (0,0) -- (2,0);
        \draw[thick, orange] (0,0.5) -- node[above]{$H-E$} (1.5,0.5) -- (2,0.5);
        \draw[dashed] (1,3) node[cross] {} -- node[below right,sloped]{\tiny $(\!1\!,\!-\!2(\!l\!+\!1\!)\!)$} (0,7);
        \draw[dashed] (1.5,2) node[cross] {} -- node[below,sloped]{\tiny $(\!1\!,\!2l\!)$}(2,1);
        \draw[thick,red] (7/4,0) node[circlecross] {} -- node[below left, sloped]{$L_{2(l+1),1}$} (1,3);
    \end{scope}

  \end{tikzpicture}
  \caption{Construction of Lagrangian $L_{2k,1}$-pinwheels in $(X_1,\omega_{h,\mu})$.}
  \label{fig:pinwheels_X_1}
\end{figure}

\subsection{Compactifications of rational homology balls}
\label{subsec:compactification}

Here is another way to produce pinwheels in $S^2\times S^2$ and $X_1$, that recovers the pinwheels constructed in the beginning of \cref{Section-1-Construction}. 
The compactification procedure discussed in this subsection was already known to Symington \cite{Sym98:Ratblo}, but we will give some more details here.
Furthermore, in this subsection we will also introduce some notions along the way that will be crucial for the Non-squeezing theorem \cref{thrm:nonsqueezing}.

For $n>1$ start by considering the almost toric base diagram $\Delta_{n,1}$, as depicted in \cref{fig:atfrationalhomologyball_n_1} and recall that the symplectic cut, originally introduced by Lerman in \cite{Ler95:cut}, fits perfectly into the almost toric setting, as explained by Evans in \cite[Section 4]{Ev23:Book}. 
We now want to compactify $B_{n,1}$ by two consecutive symplectic cuts.

\begin{figure}[ht]
  \centering
  \begin{tikzpicture}
    \begin{scope}
        \fill[opacity=0.2] 
            (0,2) -- (0,0) -- (9,2);
        \draw[thick, mid arrow] 
            (0,0) -- node[left] {\tiny $(0,1)$} 
            (0,2);
        \draw[thick, mid arrow]
            (0,0) -- node[anchor=north west] {\tiny $(n^2,n-1)$} 
            (9,2);
        \node at (1,1.5) {$\Delta_{n,1}$};
        \draw[dashed] (3,1) node[cross] {} -- node[anchor=south east] {\tiny $(n,1)$} (0,0);
    \end{scope}
  \end{tikzpicture}
  \caption{The almost toric base diagram $\Delta_{n,1}$.}
  \label{fig:atfrationalhomologyball_n_1}
\end{figure}
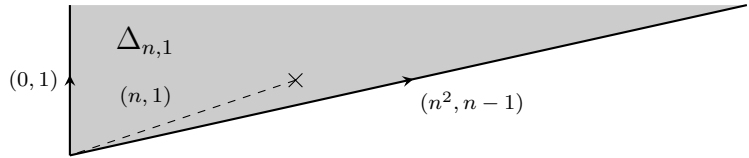

First, a horizontal symplectic cut above the node in $\Delta_{n,1}$ yields an almost toric base diagram with a orbifold singularity at the point where the horizontal line hits the slanted edge that points in the $(n^2,n-1)$-direction. 
This can be resolved by a second symplectic cut along the direction $(n+1,1)$. 
These two consecutive symplectic cuts yield the almost toric base diagram depicted in \cref{fig:atfrationalhomologyball_compactified_ellipsoids}. 
It is crucial to note that there are certain choices of sizes involved in this construction. This is the content of the following definition.

\begin{definition}
\label{Definition_rational_homology_ball_embeddings}
    For $\alpha,\beta>0$ define $\Delta_{n,1}(\alpha,\beta)$ to be the half-open triangle contained in $\Delta_{n,1}$ that is spanned by the two vectors $\alpha(n^2,n-1)$ and $\beta(0,1)$. 
    Furthermore, denote by $B_{n,1}(\alpha,\beta)$ the symplectic manifold that is defined by the almost toric base diagram $\Delta_{n,1}(\alpha,\beta)$, and for $\lambda>0$ abbreviate $B_{n,1}(\lambda)\vcentcolon=B_{n,1}(\lambda,\lambda)$. 
    By $B_{n,1}(\alpha,\infty)$ we denote the symplectic manifold associated to the half open vertical strip above the vector $\alpha(n^2,n-1)$, which we denote by $\Delta_{n,1}(\alpha,\infty)$. 
    For a pictorial description compare \cref{fig:atfrationalhomologyball_compactified_ellipsoids}.
\end{definition}

\begin{remark}
    Note that these definitions correspond to the usual definitions one makes when studying symplectic embedding problems. 
    For example, $B_{n,1}(\alpha,\beta)$ is the "rational homology ellipsoid", $B_{n,1}(\lambda)$ the "rational homology ball" and $B_{n,1}(\alpha,\infty)$ is the "rational homology cylinder". 
    In the case $n=1$ these definitions are exactly the definition of the usual symplectic ellipsoid and cylinder. 
    Embedding problems for rational homology balls and ellipsoids are studied in the forthcoming paper \cite{BrEvHaSch25}.
\end{remark}

From the definitions of the open domains we can define the compactifications that we discussed in the beginning of this section.

\begin{definition}
    For $\beta>(n-1)\alpha$ we denote by $(X_{n,1}(\alpha,\beta),\omega_{\alpha,\beta})$ the symplectic manifold that is obtained by compactifying $B_{n,1}$ via two symplectic cuts on $B_{n,1}$ as depicted in \cref{fig:atfrationalhomologyball_compactified_ellipsoids}.
    The almost toric base diagram underlying the definition of $X_{n,1}(\alpha,\beta)$ is denoted by $\Gamma_{n,1}(\alpha,\beta)$.
\end{definition}

\begin{remark}
    The condition $\beta>(n-1)\alpha$ is just a reflection of the geometry of the almost toric base diagram. 
    However, this condition will become crucial soon.
    In the following we will sometimes drop the notation of the parameters $\alpha$ and $\beta$, and denote by $X_{n,1}$ some compactification of $B_{n,1}$ for admissible parameters $\alpha$ and $\beta$ and similarly for~$\Gamma_{n,1}$. 
\end{remark}

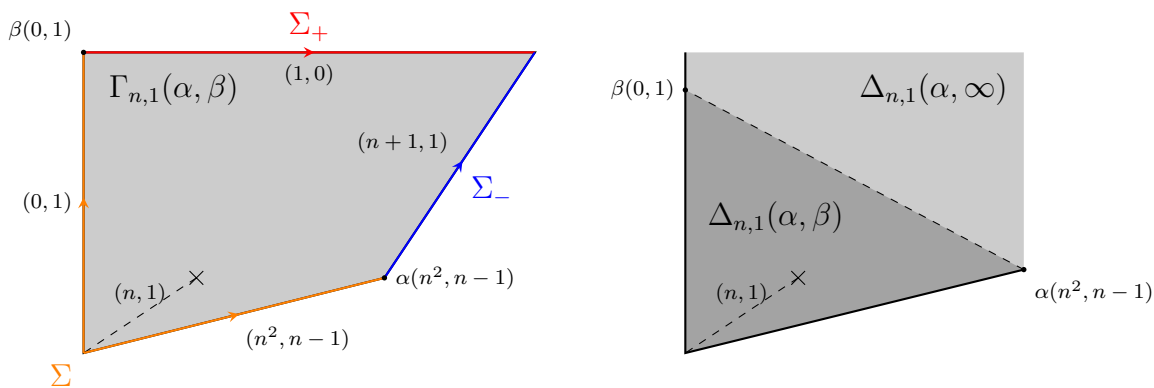
\begin{figure}[ht]
  \centering
  \begin{tikzpicture}
    \begin{scope}[shift={(-4,-1.75)}]
        \fill[opacity=0.2] 
            (0,4) -- (0,0) -- (4,1) -- (6,4);
        \draw[thick] 
            (0,4) node[anchor=south east] {\tiny $\beta(0,1)$} -- node[anchor=east] {\tiny $(0,1)$} 
            (0,0) -- node[anchor=north west] {\tiny $(n^2,n-1)$} 
            (4,1) node[anchor=west] {\tiny $\alpha(n^2,n-1)$} -- node[anchor=south east] {\tiny $(n+1,1)$} 
            (6,4) -- node[anchor=north] {\tiny $(1,0)$} 
            (0,4);
        \draw[dashed] 
            (1.5,1) node[cross] {} -- node[anchor=south] {\tiny $(n,1)$} (0,0);
        \draw[thick, red, mid arrow]
            (0,4) -- node[anchor=south] {$\Sigma_+$} 
            (6,4);
        \draw[thick, blue, mid arrow]
            (4,1) -- node[anchor=north west] {$\Sigma_-$} 
            (6,4);
        \draw[thick, orange, mid arrow]
            (0,0) node[anchor=north east] {$\Sigma$} -- (0,4);
        \draw[thick, orange, mid arrow]
            (0,0) -- (4,1);
        \fill (0,4) circle (1pt);
        \fill (4,1) circle (1pt);
        \node at (1.2,3.5) 
            {$\Gamma_{n,1}(\alpha,\beta)$};
    \end{scope}

    \begin{scope}[shift={(4,-1.75)}]
        \fill[opacity=0.2] 
            (0,4) -- (0,0) -- (4.5,1.11) -- (4.5,4);
        \fill[opacity=0.2] 
            (0,4) -- (0,0) -- (4.5,1.11) -- (0,3.5);
        \draw[thick] 
            (0,4) -- (0,0) -- (4.5,1.11);
        \draw[dashed] 
            (1.5,1) node[cross] {} -- node[anchor=south] {\tiny $(n,1)$} (0,0);
        \draw[dashed] 
            (4.5,1.11) node[anchor=north west] {\tiny $\alpha(n^2,n-1)$} -- 
            (0,3.5) node[anchor=east] {\tiny $\beta(0,1)$};
        \node at (3.3,3.5) {$\Delta_{n,1}(\alpha,\infty)$};
        \node at (1.2,1.8) {$\Delta_{n,1}(\alpha,\beta)$};
        \fill (4.5,1.11) circle (1pt);
        \fill (0,3.5) circle (1pt);
    \end{scope}
  \end{tikzpicture}
  \caption{On the left the almost toric base diagram $\Gamma_{n,1}(\alpha,\beta)$ defining the compactification of the rational homology ball $B_{n,1}$ and on the right a depiction of the various definitions around \cref{Definition_rational_homology_ball_embeddings}.}
  \label{fig:atfrationalhomologyball_compactified_ellipsoids}
\end{figure}

\begin{remark}
    This compactification procedure follows a pattern that is very similar to the symplectic cut procedure described by Evans in \cite[Chapter 9]{Ev23:Book} that yields resolutions of singularities. 
    It can be extended to general rational homology balls $B_{p,q}$. 
    We will not go into the details of this, as the general case is not needed here.
\end{remark}

We now discuss the basic topology and the symplectic geometry of the compactification~$X_{n,1}$. 
Above the toric boundary of the almost toric fibration defined via $\Gamma_{n,1}$ we find three embedded symplectic spheres $\Sigma_+,\Sigma_-$ and $\Sigma$, as depicted in \cref{fig:atfrationalhomologyball_compactified_ellipsoids}. 
In the next lemma we determine the homology classes that they define in $H_2(X_{n,1};\mathbb{Z})$.
It follows from the geometry of the almost toric base diagram that the self-intersection numbers of these spheres are $[\Sigma_+]^2=n+1, [\Sigma_-]^2=-(n-1)$ and $[\Sigma]^2=0$.

\begin{lemma}
\label{Lemma_Compactification_S2S2_X1}
    The closed manifold $X_{n,1}$ is a rational and ruled symplectic $4$-manifold.
    If $n=2k$ is even, then $X_{2k,1}$ is diffeomorphic to $X_1$ and in the standard basis we have $[\Sigma_+]=(k+1)H-kE$, $[\Sigma_-]=(1-k)H+kE$ and $[\Sigma]=H-E$.
    If $n=2k+1$ is odd, then $X_{2k+1,1}$ is diffeomorphic to $S^2 \times S^2$ and in the standard basis we have, up to the canonical swap of factors, $[\Sigma_+]=A+(k+1)B$, $[\Sigma_-]=A-kB$ and $[\Sigma]=B$.
\end{lemma}

\begin{proof}
    Since $X_{n,1}$ contains the embedded symplectic $(n+1)$-sphere $\Sigma_+$, McDuff's classification result \cite{McD:rat_rul, FrMcD96:classrat_book} shows that $X_{n,1}$ is either a scaling of $\mathbb{C}P^2$ or a blow-up of a ruled symplectic surface. 
    However, the geometry of the almost toric base diagram defining $X_{n,1}$ implies that $X_{n,1}$ is simply connected and furthermore has second Betti number $b_2(X_{n,1})=2$.\footnote{How to read of all this information from the almost toric base diagram is discussed in \cite{Sym02:Fourtwo}.}
    This shows that $X_{n,1}$ is either $(S^2\times S^2, \omega_{a,b})$ or $(X_1,\omega_{h,\mu})$.
    
    Now, we show the claim about the homology classes. 
    We will consider the case in which $n=2k$ is even. 
    The odd case follows similarly. 
    Since the intersection form of $S^2\times S^2$ is even, $X_{2k,1}$ cannot be diffeomorphic to $S^2\times S^2$. Therefore, we know that $X_{2k,1}$ is $(X_1,\omega_{h,\mu})$. 
    By \cite[Theorem 4.2]{LiLi02:sympgenus}, or a direct calculation involving the adjunction formula and the self-intersection number, we can assume that $[\Sigma_+]=(k+1)H-kE$. Now assume that $[\Sigma_-]=aH-bE\in H_2(X_1;\mathbb{Z})$. Then we have 
    \begin{equation*}
        \left\{\begin{array}{r c c c l}
            -2k+1 & = & [\Sigma_-]^2 & = & a^2-b^2\\
            -2k+3 & = & c_1(\Sigma_-) & = & 3a-b\\
            1 & = & [\Sigma_-][\Sigma_+] & = & (k+1)a-kb
        \end{array}\right.
    \end{equation*}
    where the last equation comes from that fact that $\Sigma_+$ and $\Sigma_-$ intersect once positively. 
    Solving this overdetermined system of equations immediately yields $[\Sigma_-]=(-k+1)H+kE$. 
    Repeating the same computation for $\Sigma$ shows $[\Sigma]=H-E$.
\end{proof}

\begin{remark}
    Another way to prove \cref{Lemma_Compactification_S2S2_X1} is to notice that the almost toric base diagram $\Gamma_{n,1}$ corresponds to the almost toric base diagrams considered in the beginning of this section, compare \cref{fig:pinwheel_la_b_(l+1)a}, via a reflection. 
    Going "back" through the mutations used to construct Lagrangian pinwheels shows that the almost toric fibration defined by~$\Gamma_{n,1}$ is related to a standard toric fibration of $(S^2\times S^2,\omega_{a,b})$ or $(X_1,\omega_{h,\lambda})$ via a sequence of mutations and nodel trades, which shows that the symplectic manifold defined by these fibrations are symplectomorphic.  
\end{remark}

\begin{remark}
\label{Remark_Homology_class_pinwheel_compactification}
    The homology class represented by the Lagrangian pinwheel appearing as a visible Lagrangian over the branch cut in $X_{n,1}$ can be deduced from the almost toric base diagram \cref{fig:atfrationalhomologyball_compactified_ellipsoids} and \cref{Lemma_Compactification_S2S2_X1}. 
    In the case $n=2k+1$ the $L_{2k+1,1}$-pinwheel represents the homology class $A+kB \in H_2(S^2\times S^2;\mathbb{Z}_{2k+1})$ and in the case $n=2k$ the $L_{2k,1}$-pinwheel represents the homology class $kH +(k+1)E \in H_2(X_1;\mathbb{Z}_{2k})$.
\end{remark}

The next step is to investigate the symplectic form that we obtain by the compactification procedure. By a nodal slide we can always make the branch cut as short as we like and this will only change the symplectic manifold defined by the almost toric base diagram by a symplectomorphism \cite{Sym02:Fourtwo}. 
Therefore, we will assume in the following that the branch cut is negligibly short.

\begin{lemma}
\label{lemma_computing_sizes_compactification}
    The compactification given by the symplectic manifold $(X_{2k+1,1}(\alpha,\beta),\omega_{\alpha,\beta})$ is symplectomorphic to $(S^2\times S^2,\omega_{a,b})$, where 
    \begin{equation*}
        a=(k+1)\beta - k\alpha \text{ and } b=\beta + \alpha,
    \end{equation*}
    whereas the symplectic manifold $(X_{2k,1}(\alpha,\beta),\omega_{\alpha,\beta})$ is symplectomorphic to $(X_1,\omega_{h,\mu})$, where
    \begin{equation*}
        h=(k+1)\beta - (k-1)\alpha \text{ and } \mu=k\beta - k\alpha.
    \end{equation*}
\end{lemma}

\begin{proof}
    The point in which the edges corresponding to the spheres $\Sigma_+$ and $\Sigma_-$ intersect in the almost toric base diagram $\Gamma_{n,1}(\alpha,\beta)$ is
    \begin{equation*}
        \big((n+1)\beta + \alpha,\beta\big).
    \end{equation*}
    Hence
    \begin{equation*}
        \omega_{\alpha,\beta}(\Sigma_+)=(n+1)\beta + \alpha \quad\text{and}\quad \omega_{\alpha,\beta}(\Sigma)=\alpha+\beta.
    \end{equation*}
    The claim now follows readily from \cref{Lemma_Compactification_S2S2_X1}.
\end{proof}

\begin{remark}
\label{remark:sizes_compactification}
    We chose to use information about $\Sigma_+$ and $\Sigma$ to calculate the weights of the symplectic forms in \cref{lemma_computing_sizes_compactification}. 
    We could have used another pair of information to calculate the forms. 
    Such a pair of information is, for example, given by the fact that the total volume of $(X_{n,1}(\alpha,\beta),\omega_{\alpha,\beta})$ is equal to $\frac{n+1}{2}\beta^2 + \alpha\beta - \frac{n-1}{2}\alpha^2$ and that the area of $\Sigma_-$ is $\beta-(n-1)\alpha$, which both can be easily calculated from the geometry of the almost toric base diagram. 
    These two quantities will play a role in the proof of \cref{thrm:nonsqueezing}. 
\end{remark}

\section{Obstructing liminal Pinwheels}
\label{sec:obstruction}

In this section we prove the obstructive parts of \cref{thrm:A} and \cref{thrm:B}. 
For the rest of this section $L_n$ will denote a liminal pinwheel and we assume that $n\geq 3$. 

Recall that by \cref{Lemma_Rational_Blowup_in_S2S2_X1} blowing up $L_n$ produces the manifold $X_n$. The main point of this section is to determine the symplectic manifold $(X_n,\omega_{h,\boldsymbol{\mu}})$. 
In particular, we want to determine the cohomology class $\omega_{h,\boldsymbol{\mu}}$, as we can translate the obstructions of Seiberg--Witten theory on $[\omega_{h,\boldsymbol{\mu}}]\in H^2_{dR}(X_n;\mathbb{R})$, such as exceptional classes having positive $\omega_{h,\boldsymbol{\mu}}$-area, to obstructions to the embeddings of liminal pinwheels and the corresponding rational homology balls.

The idea is the following. 
Assuming that $L_n$ is a liminal pinwheel, we determine the homological orthogonal complement to $[L_n]$, which we denote by $\mathcal{W}_n$. 
Now, rationally blowing up the liminal pinwheel $L_n$ produces the manifold $X_n$ and introduces a $C_n$-configuration, as discussed in \cref{Section_Blowing_Up_Pinwheels}. 
This $C_n$-configuration consists of $n-1$ symplectic spheres. 
Therefore, two more pieces of information are needed to determine the cohomology class $[\omega_{h,\boldsymbol{\mu}}]\in H^2_{dR}(X_n;\mathbb{R})$. 
Denoting the homological configuration that corresponds to $\mathcal{W}_n$ under the rational blow-up/blow-down by $\mathcal{D}_n$, determining the homological configuration $\mathcal{F}_n=\{\mathcal{C}_n,\mathcal{D}_n\} \subseteq H_2(X_n;\mathbb{Z})$ allows us to determine the cohomology class $[\omega_{h,\boldsymbol{\mu}}]$ completely.

This section relies heavily on results of Li and Li \cite{LiLi02:sympgenus} on characterizing homology classes represented by smoothly embedded spheres. 
The \textit{liminal} condition introduced in  \cref{def:liminal_pinwheel} is exactly the condition that allows to use the results of \cite{LiLi02:sympgenus} in the blown up space.

After determining the homology classes in the complement of the $C_n$-configuration introduced by the rational blow-up, routine calculations on homology also determine the homology classes of the $C_n$-configuration itself. 
Finally, everything is put together in Theorem \ref{thrm:all_homology_classes}
and Corollary \ref{cor: mu_i periods} where we compute the cohomology class $\omega_{h,\boldsymbol{\mu}}$ with respect to a standard cohomological basis of $X_n$.

\subsection{Preparations}

First, we want to compute the homology classes admitting representatives disjoint from the pinwheel, since these classes will also persist under the rational blow-up of the pinwheel. 
In Theorem \ref{thrm:pinwheel_hom/orth_comp} of Appendix \ref{Appendix:orthogonal complement}, we show that these classes are precisely the classes that pair trivially modulo $n$ with the class of the pinwheel and so it is straightforward to compute them: 

\begin{lemma}\label{lem:hom_complement}
For $n=2k+1$, the orthogonal complement, with respect to the intersection product, of $[L_n]=A+kB \in H_2(S^2 \times S^2; \mathbb{Z}_{2k+1})$ is generated by the classes
\[W_1\vcentcolon= A+(k+1)B, W_2\vcentcolon= 2A+B \subseteq H_2(S^2 \times S^2; \mathbb{Z}_{2k+1}).\]
For $n=2k$, the orthogonal complement of $[L_n]=kH+(k+1)E\in H_2(X_1; \mathbb{Z}_{2k})$ is generated by 
\[W_1\vcentcolon=(k+1)H-kE,W_2\vcentcolon=2H \subseteq H_2(X_1; \mathbb{Z}_{2k}).\]
We denote the orthogonal complement by $\mathcal{W}_{n}\vcentcolon=\langle W_1,W_2\rangle$.
\end{lemma}

\begin{remark}
Notice that in either case, the classes that generate $\mathcal{W}_n$ are represented by smoothly embedded spheres, see \cite[Theorem C]{LiLi02:sympgenus}.
\end{remark}

As we explain in Appendix \ref{Appendix:orthogonal complement}, we can represent the classes $W_1,W_2$ by embedded oriented submanifolds disjoint from $L_n$. 
Therefore, we can find an embedding of the sublattice $\mathcal{W}_n=\langle W_1,W_2\rangle$ in $H_2(X_n;\mathbb{Z})$. 
Determining how this sublattice looks like in $H_2(X_n;\mathbb{Z})$ means finding its generators in $H_2(X_n;\mathbb{Z})$. 
For $j=1,2$ let $D_j$ be the image of $W_j$ in $H_2(X_n;\mathbb{Z})$.
For a worked out example of this construction in the specific case of a Lagrangian $\mathbb{R}P^2$ compare \cite[Lemma 2.2.1]{SmSh20}.
We now want to determine $D_1$ and $D_2$.
This is exactly where the liminality condition comes into play, and to do so we will need some auxiliary results.

\begin{theorem}\label{theorem:LiLi}\emph{(\cite[Theorem D]{LiLi02:sympgenus})}
Consider two classes, $S,S'\in H_2(X_n;\mathbb{Z})$ such that they are both represented by smooth spheres, $S^2=S'^2\geq -1$, they are either both \textit{characteristic}, i.e.\ Poincaré dual to an integral lift of the second Stiefel--Whitney class $w_2$, or \textit{ordinary}, i.e.\ not characteristic, and they have the same divisibility.
Then, there exists a diffeomorphism $\phi: X_n\rightarrow X_n$ such that $\phi_* S=S'$.
\end{theorem}

We now show that we can indeed apply Theorem \ref{theorem:LiLi} to one of the classes $D_1$ or $D_2$. 
For this we will show that one of the $D_i$ can be assumed to be primitive and ordinary.

\begin{lemma}\label{lemma:conditions_Di}
Let $\mathcal{D}_n=\langle D_1,D_2\rangle$ be the homological configuration in $H_2(X_n;\mathbb{Z})$, as above, which is disjoint from a $C_n$-configuration coming from blowing up a liminal $L_{n}$-pinwheel. 
By the liminality condition, there exists an embedded sphere representing, say $D_1$, that is disjoint from the $C_n$-configuration.
Then this $D_1$ is ordinary and both $D_1$ and $D_2$ primitive.
\end{lemma}

\begin{proof}
We start by showing that $D_1$ is ordinary. 
Recall that if a class $D$ is characteristic then, by definition, $D\cdot C\equiv C\cdot C \mod{2}$ for every homology class $C$. 
So to show that a class is ordinary it is therefore enough to show that it pairs trivially with a class of odd self-intersection. 
Recall that we assumed that $D_1$ is represented by an embedded sphere. 
Then, by \cite[Corollary 4.1]{LiLi02:sympgenus}, there exists a symplectic form $\omega'$, that is possibly different from the one provided by the rational blow-up, for which $D_1$ can be represented by an $\omega'$-symplectic sphere, which we denote by $D'$. 
Let $\nu D'$ be a standard symplectic neighborhood of $D'$, such that $X_n-\nu D'$ is a convex symplectic filling of $\partial( \nu D')$, which is the lens space $L(p,1)$ where $p$ denotes the self-intersection number of $D'$. 
By McDuff's classification \cite[Theorem 1.7]{McD:rat_rul} of minimal symplectic fillings, $X_n-\nu D'$ cannot be minimal since all minimal symplectic fillings have second Betti number $b_2=0$ or $b_2=1$, while $X_n-\nu D'$ has $b_2\geq 2$.
Therefore $X_n-\nu D'$ carries a symplectic exceptional sphere $E$. 
By construction $D'\cdot E=0$ and thus $D_1$ cannot be characteristic.

The proof for the case in which $D_2$ is represented by an embedded sphere disjoint from the $C_n$-configuration follows analogously.

To show the primitivity statement, we suppose that $D_j$ is not primitive, i.e.\ that there exists some class $P\in H_2(X_n;\mathbb{Z})$ and some $m\in \mathbb{Z}$ with $\lvert m\rvert\geq2$ such that $D_j=mP$. 
Since $D_j\cdot S_i=0$, where the $S_i$ correspond to the spheres in the $C_n$-configuration, we have that $P\cdot S_i=0$, which means that $P$ is in the homological complement of the $L_{n,1}$-pinwheel. This contradicts the fact that the $D_j$ correspond to the $W_j$ under the rational blow-down.
\end{proof}

\begin{remark}
As the proof of \cite[Proposition 1.4.]{BoLiWu13} shows, one can construct examples where an ordinary spherical class becomes characteristic after rationally blowing up. 
In our case, we will see that the class $[\Sigma_-]$, that is introduced by compactifying rational homology balls, as described in \cref{subsec:compactification}, becomes $(3-n)H+(n-3)E_1+E_2+\cdots+E_{n-1}-E_n$ which is also characteristic.
\end{remark}

\begin{remark}
Note that the class $D_1$ cannot be characteristic for numerical reasons: since $w_2(X_n)=H+\sum_{i=1}^n E_i$, we would have $n+1=D_1^2=w_2(D_1)\equiv 0 \mod 2$ as well as $n=w_2(S_0)=D_1\cdot S_0\equiv 0 \mod 2$ which implies a contradiction on the parity of $n$. 
\end{remark}

\cref{lemma:conditions_Di} combined with \cref{theorem:LiLi} implies: 

\begin{proposition}\label{prop: determining d_i}
Let $L_n$ be a liminal pinwheel, $\mathcal{W}_n =\langle W_1,W_2 \rangle$ its orthogonal complement, and let $X_n$ be the manifold obtained by blowing up $L_n$. 
Then, depending on whether $D_1$ or $D_2$ carries a smooth sphere, there exists a diffeomorphism $\phi$ or $\psi$ of $X_n$ such that
\begin{enumerate}
    \item 
        $\phi_*D_1= nH-(n-1)E_1-E_2-\cdots -E_{n-1}$ if $D_1$ carries an embedded sphere, or, 
    \item 
        $\psi_*D_2=3H-2E_1-E_n$ if $D_2$ carries an embedded sphere.
\end{enumerate}
\end{proposition}

\begin{remark}\label{remark:different_expression_D1}
Even though the homology classes $D_1$ and $D_2$ of \cref{prop: determining d_i} have seemingly different expressions, after a change of coordinates we may assume that $D_1$ lives in the class $(k+1)H-kE_1$ if $n=2k+1$, or $(k+1)H-kE_1-E_n$ if $n=2k$.
\end{remark}

\begin{remark}
Here is an alternative way to prove \cref{prop: determining d_i}. 
Applying \cite[Theorem 4.2.]{LiLi02:sympgenus} directly to the class $D_i$ that is represented by an embedded sphere one can assume that $D_i$ is equivalent to a specific class and then show that this class is Cremona equivalent to one of the two classes in \cref{prop: determining d_i}.
\end{remark}

Surprisingly, just determining one of the homology classes of the $D_i$ enables us to determine the homology class $\mathcal{C}_n$ of the $C_n$-configuration.
For the convenience of the reader, we will break the proof into two lemmas. 
In what follows we will only provide proofs for the case in which $D_2$ is represented by an embedded sphere, as the $D_1$ case is barely any different, as also mentioned in \cref{remark:different_expression_D1}.

Recall that the configuration $C_n$ consists of $S_0$, a $(-n-2)$-sphere, and then $(-2)$-spheres $S_i$ for $1\leq i \leq n-2$. 
We start by determining the sub-chain formed by the $(-2)$-spheres.

\begin{lemma}[Determining the $(-2)$-spheres]\label{lemma:Ci_config_2spheres}
Let $S_i$ for $i=1,\dots, n-2$ form the maximal sub-chain of a $C_n$-configuration that consists of $(-2)$-spheres. Moreover, assume that the $C_n$-configuration is homologically orthogonal to
\[D_1=nH-(n-1)E_1-E_2-\cdots -E_{n-1}\text{ or } D_2=3H-2E_1-E_n.\]
Up to diffeomorphism fixing the corresponding $D_i$, we may assume that the homology classes are
\[S_1=E_{n-2}-E_{n-1},\ldots,S_{n-3}=E_2-E_3\text{ and } S_{n-2}=H-E_1-E_2-E_n.\]
\end{lemma}

\begin{proof}
We will assume that the configuration $C_n$ is homologically orthogonal to $D_2$. 
The argument for the case in which $C_n$ is homologically orthogonal to $D_2$ follows analogously.
Fix $i$ and assume that $S_i=aH-\sum_{j=1}^{n} b_j E_j\in H_2(X_n;\mathbb{Z})$. 
The orthogonality $S_i\cdot D_2=0$ implies that $3a-2b_1-b_n=0$, which means that
\begin{equation}\label{eq:basiceq}
a=\frac{2b_1+b_n}{3}.
\end{equation}
Since $S_i^2=-2$ we have $a^2-\sum_{1=j}^n b_j^2=-2$. Combining this with Equation \ref{eq:basiceq} we get, after reordering and distributing the terms,
\begin{equation}\label{eq:formforci}
    4b_1^2+4b_n^2+(b_1-2b_n)^2+9\sum_{j=2}^{n-1}b_j^2=18.
\end{equation}
From \cref{eq:formforci} it is easy to see that  $-1\leq b_j\leq 1$ for $1\leq j\leq n$ and also that at most three of the $b_j$ can be non-vanishing at the same time. 
In fact, the only classes with three non-vanishing coefficients, called \textit{ternary classes} in \cite[Lemma 2.3.]{Ev10}, are
\[S_i=\pm(H-E_1-E_j-E_n)\text{ for some } 2\leq j\leq n-1.\]
We will call the ternary classes \textit{positive} or \textit{negative}, depending on the sign in this equation.
The only other possible integer solutions to \cref{eq:formforci}, together with the adjunction formula, lead to \textit{binary classes}, which are given by
\[S_i=E_j-E_{r}\text{ for } 2\leq j,r\leq n-1.\]

Any chain\footnote{By a chain we mean here a finite sequence of homology classes such that the intersection number of a class with its direct neighbors is equal to $+1$ and $0$ with all the other classes.} formed by binary classes must have length less or equal to $n-3$. 
This is a consequence of the intersection pattern prescribed for the chain.
Therefore $\mathcal{C}_n$, the homological configuration associated to the $C_n$-configuration, must carry at least one ternary class.
However, $\mathcal{C}_n$ can carry at most two ternary classes, again for intersection reasons.
Moreover, one of the ternary classes must be positive and the other negative and they have to be adjacent. 
Let us now show that any chain $\mathcal{S}=(S_1,\dots, S_{n-2})$, consisting of binary and ternary classes as above, can be mapped via a diffeomorphism to a chain for which the only ternary class is $S_{n-2}=H-E_1-E_2-E_n$.
We distinguish two cases:
\begin{itemize}
    \item[1)]
    We start by considering a chain $S_1,\dots, S_{n-2}$ that has only one ternary class that is either positive or negative. 
    Two different binary classes intersecting a ternary class $\pm(H-E_1-E_j-E_n)$ non-trivially must share the class $\mp E_j$, meaning that they have a nontrivial intersection. 
    This means that the ternary class can only be either $S_1$ or $S_{n-2}$. 
    We distinguish these two subcases:
        \begin{itemize}
            \item[a)] 
                Suppose that the unique ternary class is $S_1$. 
                Then, after relabeling the $E_i$, we can assume that $S_1=\pm(H-E_1-E_n-E_2)$ and that $S_{i}=\pm (E_{i}-E_{i+1})$ for $2\leq i\leq n-2$. 
                Therefore, Dehn-twisting around $B=\pm (H-E_1-E_n-E_{n-1})$ makes $S_{n-2}$ positive ternary, which is what we wanted. 
                Note that the class $D_2$ is fixed under this Dehn-twist and the remaining binary classes of the chain stay binary under this Dehn-twist. 
                After a change of indices this yields the desired homological configuration.
            \item[b)]
                Suppose that the unique ternary class is $S_{n-2}$. 
                If it is positive, we are done. 
                If it is negative, by the same reasoning as above, we can Dehn-twist to find a chain that has $S_1$ positive and ternary and thus we fall back to the Case a).
        \end{itemize}
    \item[2)] 
    Now we consider the case in which the chain $S_1,\dots,S_{n-2}$ carries two ternary classes. 
    As mentioned before one of them has to be positive and the other negative.
        \begin{itemize}
            \item[a)] Suppose that the ternary classes are $S_{n-3}$ and $S_{n-2}$.
            Dehn-twisting around $S_{n-2}$ keeps $S_{n-2}$ ternary and makes $S_{n-3}$ binary. Therefore we fall back to Case 1).
            \item[b)] Similar to Case 1a), if $S_{n-2}$ is not ternary, then after relabeling the $E_i$ we have that $S_{n-2}=E_{n-1}-E_{n-2}$ and thus Dehn-twisting along $H-E_1-E_n-E_{n-1}$ makes the class $S_{n-2}$ ternary and so we fall back to one of the other cases.
        \end{itemize}
\end{itemize}

We have thus shown that, after a diffeomorphism preserving the class $D_2$, any chain $S_1,\dots,S_{n-2}$
is diffeomorphic to one with a single positive ternary class and that this ternary class is $S_{n-2}$.
\end{proof}

We now move to the determination of the remaining sphere of the configuration $C_n$, namely $S_0$.

\begin{lemma}[Determining the $(-n-2)$-sphere]\label{lemma:C0_config}
Suppose that there is a $C_n$-configuration in $(X_n,\omega)$ such that $\mathcal{C}_n$ is homologically orthogonal to
\[D_1=nH-(n-1)E_1-E_2-\cdots -E_{n-1} \text{ or } D_2=3H-2E_1-E_n\]
and such that the $(-2)$-spheres in the configuration $C_n$ represent the classes
\[S_1=E_{n-2}-E_{n-1},\ldots,S_{n-3}=E_2-E_3 \text{ and } S_{n-2}=H-E_1-E_2-E_n.\]
Then, the $(-n-2)$-sphere $S_0$ represents the class
\[S_0=-2H+3E_1-E_2-\cdots-E_{n-2}.\]
\end{lemma}

\begin{proof}
As in the proof of \cref{lemma:Ci_config_2spheres} we assume that $\mathcal{C}_n$ is homologically orthogonal to $D_2$. The case in which $\mathcal{C}_n$ is homologically orthogonal to $D_1$ follows along the same lines.
Suppose that $S_0=aH-\sum_{j=1}^n b_j E_j$. 
The conditions $S_0 \cdot S_1=1$, $S_0\cdot S_i=0$ for $2\leq i\leq n-3$ and $S_0 \cdot S_{n-2}=0$ imply
\begin{equation}\label{eq:lem_5_11_def_b}
    1=b-b_{n-1},\quad b\vcentcolon =b_2=\cdots=b_{n-2} \quad\text{and}\quad 0=a-b_1-b-b_n
\end{equation}
and the equation obtained by using $S_0\cdot D_2=0$ is
\begin{equation}\label{eq:lem_5_11_a}
    a=\frac{2b_1+b_n}{3}.
\end{equation}
In addition, combining the third equation of \cref{eq:lem_5_11_def_b} and \cref{eq:lem_5_11_a} means that $\frac{2b_1+b_n}{3}=b_1+b+b_n$, which means 
\begin{equation}\label{eq:lem_5_11_b_1}
    b_1=-3b-2b_n.
\end{equation}
The adjunction formula for the symplectic sphere $S_0$ reads $-n=c_1(S_0)=3a-\sum_{j=1}^n b_j$ which, after using \cref{eq:lem_5_11_def_b,eq:lem_5_11_a,eq:lem_5_11_b_1}, implies
\begin{equation}\label{eq:lem_5_11_b_1_updated}
    b_1=(n-2)b-(n+1).
\end{equation}
Using this equation for $b_1$ in \cref{eq:lem_5_11_b_1} shows
\begin{equation}\label{eq:lem_5_11_2b_n}
    2b_n=-3b-b_1=n+1-(n+1)b.
\end{equation} 
Now we want to use the equation on the coefficients that is coming from the prescribed self-intersection number, i.e.\ $-n-2=S_0^2=a^2-\sum_{j=1}^n b_j^2$.
Plugging \cref{eq:lem_5_11_def_b} and \cref{eq:lem_5_11_a} into this and regrouping the terms yields 
\begin{equation*}
4b_1^2+4b_n^2+(b_1-2b_n)^2+9(n-3)b^2 +9b_{n-1}^2=9(n+2),
\end{equation*}
which, after plugging in \cref{eq:lem_5_11_b_1_updated,eq:lem_5_11_def_b,eq:lem_5_11_2b_n}, reads
\begin{equation}\label{eq:basiceq_C0}
    \begin{split}
        4\Big((n-2)b-(n+1)\Big)^2&+4\Big(n+1-(n+1)b\Big)^2+\Big((2n-1)b-2(n+1)\Big)^2\\
        &+9(n-3)b^2+9(b-1)^2=9(n+2).
    \end{split}
\end{equation}
It is easy to see that for $b^2\geq 2$ the left-hand side of \cref{eq:basiceq_C0} is bigger than the right-hand side.
Furthermore, for $b=0$ \cref{eq:basiceq_C0} reduces to $9(n+1)=12(n+1)^2$, which has no integer solutions, and for $b=-1$ it reduces to
\[4(2n-1)^2+4(n+1)^2+(4n+1)^2+9(n-3)^2+36=9(n+2),\]
which also does not have integer solutions, as for $n\geq 3$ already the term $(4n+1)^2$ is greater than the right-hand side.
Finally, for $b=1$ \cref{eq:basiceq_C0} is indeed satisfied and using \cref{eq:lem_5_11_2b_n,eq:lem_5_11_b_1_updated,eq:lem_5_11_a} shows
\[S_0=-2H+3E_1-E_2-\cdots-E_{n-2}.\]
\end{proof}

With this preparational work in place we are now able to prove:

\begin{theorem}\label{thrm:all_homology_classes}
Let $L_n$ be a liminal pinwheel and let $\mathcal{W}_n$ be the set of homology classes in the orthogonal complement. 
The manifold obtained by blowing up the pinwheel is the rational symplectic manifold $(X_n,\omega)$ and therein the homology classes generating $\mathcal{D}_n$ are, up to diffeomorphism, given by
\[D_1=nH-(n-1)E_1-E_2-\cdots -E_{n-1} \text{ and } D_2=3H-2E_1-E_n\]
and the homological configuration $\mathcal{C}_n$ of the chain $C_n$ is given by
\[S_0=-2H+3E_1-E_2-\cdots-E_{n-2},\quad S_1=E_{n-2}-E_{n-1},\quad\ldots\quad, S_{n-3}=E_2-E_3\]
\[\quad S_{n-2}=H-E_1-E_2-E_n.\]
\end{theorem}

\begin{proof}
By \cref{lemma:conditions_Di}, the generators of $\mathcal{W}_n$, which after the blow-up become $\mathcal{D}_n$, satisfy all the conditions of Theorem \ref{theorem:LiLi}. 
Thus, by the liminality condition, we can determine $D_1$ or $D_2$. As before we assume that $D_2$ is represented by an embedded sphere and fix $D_2=3H-2E_1-E_n$. The case in which $D_1$ is represented by an embedded sphere follows analogously.
Making such a choice, \cref{lemma:Ci_config_2spheres} and \cref{lemma:C0_config} then determine the homological configuration $\mathcal{C}_n$ completely.

The last thing to determine is $D_1$. 
This can be done with a computation similar to the one carried out in \cref{lemma:C0_config}.
\end{proof}

With the above theorem we can, in principle, completely describe the cohomology class of $\omega$ with respect to the areas of the homology classes $D_i$ and $S_j$. 
It turns out that, for our purposes we will only need to determine the periods of the last two exceptional spheres $E_i$.

\begin{corollary}\label{cor: mu_i periods}
Let $L_n$ be a liminal pinwheel and suppose that $d_1$ and $d_2$ are the symplectic areas of the $W_1$ and $W_2$ homology class. 
Consider the symplectic manifold $(X_n,\omega)$ obtained by blowing up $L_n$, where the spheres $S_i$ for $i=0,\dots,n-2$ have positive symplectic area~$c_i$. 
Then the two exceptional classes $E_{n-1}$ and $E_{n}$ have areas
\begin{align}\label{eq:mu_n area}
    \begin{split}
        \mu_{n}&=\frac{1}{n^2}\Big((n+2)d_1-(n+1)d_2-\sum_{i=0}^{n-2}\big(1+i(n+1)\big)c_i\Big)\\
        \mu_{n-1}&=\frac{1}{n^2}\Big((n+2)d_2-4d_1-\Big((n-2)c_0+(n+2)\sum_{i=1}^{n-2}(n-(i+1))c_i\Big)\Big).
    \end{split}
\end{align}
\end{corollary}

\begin{proof}
Since the $D_i$ and $S_j$ form a basis of $H_2(X_n;\mathbb{R})$ there is a linear transformation mapping these classes to the standard basis. 
A tedious but straightforward calculation gives the desired result. 
For more details see \cref{Appendix:solving_mu}.
\end{proof}

\subsection{Necessity of the inequalities}

Now we use \cref{thrm:all_homology_classes} to show that the existence of liminal pinwheels implies the inequalities on the periods of the symplectic forms claimed in \cref{thrm:A} and \cref{thrm:B}. 

\begin{proposition}[Main calculation]
\label{prop:main_calculation}
Suppose that $(X_n,\omega_{h,\boldsymbol{\mu}})$ carries an $\mathcal{F}_n = \{\mathcal{C}_n,\mathcal{D}_n\}$ configuration, where $\mathcal{C}_n$ is represented by an $C_n$-configuration of symplectic spheres and one of the $D_i$ of $\mathcal{D}_n$ satisfies the conditions of \cref{thrm:all_homology_classes}. 
Moreover, assume that $(X_n,\omega_{h,\boldsymbol{\mu}})$ and the configuration actually stem from blowing up a liminal pinwheel.
Blowing down the $C_n$-configuration leads to
\begin{enumerate}
    \item $(S^2\times S^2,\omega_{a,b})$ with $\frac{a}{k+1}<b<2a$ when $n=2k+1$ and
    \item $(X_1,\omega_{h,\mu})$ with $\mu<\frac{k}{k+1}h$ when $n=2k$.
\end{enumerate}
\end{proposition}

\begin{proof}
Denote by $c_i$ and $d_i$ the $\omega_{h,\boldsymbol{\mu}}$-periods of the classes of $\mathcal{F}_n$. 
These homology classes constitute an $\mathbb{R}$-basis of $H_2(X_n;\mathbb{R})$ and thus we can express the periods of the standard basis, namely the $h,\mu_i$, in terms of the periods of $\mathcal{F}_n$.\\
\cref{cor: mu_i periods} gives 
\begin{equation*}
    \mu_{n-1}=\frac{(n+2)d_2-4d_1}{n^2}- \tilde{c}_{n-1} \quad\text{and}\quad
    \mu_{n}=\frac{(n+2)d_1-(n+1)d_2}{n^2}-\tilde{c}_{n}
\end{equation*}
where $\tilde{c}_{n-1}$ and $\tilde{c}_{n}$ are some \textbf{positive} linear combination of the areas $c_i$, which are themselves positive. 
The positivity of $\tilde{c}_{n}$ and $\tilde{c}_{n-1}$, together with the positivity of the $\mu_i$'s implies
\begin{equation}\label{eq:ineqfordi}
    (n+2)d_2-4d_1>0 \quad\text{and}\quad (n+2)d_1-(n+1)d_2>0.
\end{equation}
When $n=2k+1$, we have that $d_1=a+(k+1)b$ and $d_2=2a+b$ and so \cref{eq:ineqfordi} implies exactly 
\[2a>b \quad\text{and}\quad (k+1)b>a,\]
whereas when $n=2k$, we have $d_1=(k+1)h-k\mu$ and $d_2=2h$ and \cref{eq:ineqfordi} implies
\[0<4k\mu \quad\text{and}\quad (k+1)\mu<kh,\]
where the first inequality is redundant.
\end{proof}

\begin{remark}
    Appealing to \cite{PaShi23} one can actually get away with just assuming that the appropriate configuration exists in $(X_n,\omega_{h,\boldsymbol{\mu}})$. 
    However, since this is not needed in this paper we refrain from discussing this in detail.
\end{remark}

\cref{prop:main_calculation} then implies:

\begin{theorem}[\cref{thrm:A}]\label{thrm:s2xs2_obstruction}
If $L_{2k+1,1}$ is a liminal pinwheel in $(S^2\times S^2,\omega_{a,b})$ then
\[2a>b>\frac{a}{k+1}.\] 
\end{theorem}

and

\begin{theorem}[\cref{thrm:B}]\label{thrm:x1_obstruction}
If $L_{2k,1}$ is a liminal pinwheel in $(X_1,\omega_{h,\mu})$ then 
\[\mu<\frac{k}{k+1}h.\]
\end{theorem}

\section{Corollaries of \texorpdfstring{\cref{thrm:A}}{Theorem A} and \texorpdfstring{\cref{thrm:B}}{Theorem B}}
\label{sec:corollaries}

\subsection{Khodorovskiy's smooth embeddings of rational homology balls}

In \cite{Kho14:Embratball} Khodorovskiy proves the following two theorems.

\begin{theorem}(\cite[Theorem 1.2]{Kho14:Embratball})\label{thrm:Kho_-n-1}
Assume that $n\geq2$ is an integer and let $V_{-n-1}$ be a neighborhood of a sphere with self-intersection number $-n-1$.
Then, there exists a smooth embedding $B_{n,1} \hookrightarrow V_{-n-1}$.
\end{theorem}

\begin{theorem}(\cite[Theorem 1.3]{Kho14:Embratball})\label{thrm:Kho_-4_and_2}
Let $V_{-4}$ be a neighborhood of a sphere with self-intersection number $-4$ and assume that $n\geq 3$ is an odd integer. Then, there exists a smooth embedding $B_{n,1} \hookrightarrow V_{-4}$.

Moreover, if $n\geq 2$ is an even integer, there exists a smooth embedding $B_{n,1} \hookrightarrow B_{2,1}\#\mathbb{C}P^2$.
\end{theorem}

\cref{thrm:A} and \cref{thrm:B} readily imply that these embeddings cannot be realized symplectically, except for embeddings into $B_{2,1}\#\mathbb{C}P^2$. 
Note that the "soft" topological obstruction, namely the formula for the self intersection of Lagrangian pinwheels, proven in \cref{Lag_pinwheels_self_int} does not exclude the existence of symplectic embeddings. 

\begin{corollary}[Theorem \ref{thrm:Kho}]
Let $(V_{-n},\tau)$ be a standard symplectic neighborhood of a symplectic $(-n)$-sphere. Then:
\begin{enumerate}
    \item 
        If $n=2k+1$ for $k\geq 1$, there are no symplectic embeddings of $B_{2k+1,1}$ in $V_{-2k-2}$ or $V_{-4}$, even though smooth ones exist.
    \item 
        If $n=2k$ there are no symplectic embeddings of $B_{2k,1}$ in $V_{-2k-1}$, even though smooth ones exist.
\end{enumerate}
\end{corollary}

\begin{proof}
We prove the case $n=2k+1$. Assume that $B_{2k+1,1}$ symplectically embeds into $V_{-2k-2}$. Compactifying $V_{-2k-2}$ as in Lemma \ref{lemma:compactification} yields $(S^2\times S^2,\omega_{a,b})$. 
Fixing the core sphere of $V_{-2k-2}$ to represent the class $A-(k+1)B \in H_2(S^2 \times S^2;\mathbb{Z})$ in the compactification means that the core pinwheel of $B_{2k+1,1}$ represents the $\mathbb{Z}_{2k+1}$-class $A+kB$.\footnote{This follows easily from the soft obstruction in \cref{remark:pinwheels_soft_obstruction}.}
Now, the period inequalities we get from Lemma \ref{lemma:compactification} are \textbf{exactly} the opposite to those coming from Theorem \ref{thrm:A}, and thus contradicting the existence of the original symplectic embedding.

The case of an embedding into $V_{-4}$ as well as the case $n=2k$ are completely analogous.
\end{proof}

\begin{remark}
    The fact that there are symplectic embeddings $B_n \hookrightarrow B_2\#\mathbb{C}P^2$ for $n\geq 2$ even, for certain symplectic forms on $ B_2\#\mathbb{C}P^2$, follows from our constructions outlined in \cref{Section-1-Construction}. 
    The complement of the quadric in $\mathbb{C}P^2$ is symplectomorphic to $B_2$ and so the complement of a conic in $X_1$ is symplectomorphic to $B_2\#\mathbb{C}P^2$. 
\end{remark}

\subsection{A non-squeezing result}

Here is the non-squeezing theorem for rational homology balls, analogous to Gromov's famous non-squeezing theorem.

\begin{theorem}\label{thrm: non-squeezing}
    There exists a symplectic embedding $\iota:B_{n,1}(1)\hookrightarrow B_{n,1}(\alpha,\infty)$ if and only if $\alpha \geq 1$.
\end{theorem}

The proof follows along the same lines as the proofs before, i.e.\ as the proofs of Theorems \ref{thrm:A},\ref{thrm:B} and \ref{thrm:Kho}, the difference now being that we will use quantitative information on the sizes involved in the blow-up, i.e.\ we will specify the areas of the chain of spheres introduced through the rational blow-up. Understanding exactly how these areas can be chosen will give us the constraint on $\alpha$.

\begin{proof}
As in the proof of Gromov's non-squeezing theorem the constructive direction of the proof is trivial.
Hence, we prove the obstructive part of the theorem. 

Assume that $\alpha > 0$ and that there is a symplectic embedding $\iota:B_{n,1}(1)\hookrightarrow B_{n,1}(\alpha,\infty)$. 
The image of this embedding has to be bounded, meaning that there is a $\beta>0$ such that $B\vcentcolon=\overline{\iota(B_{n,1}(1))}\subseteq P_{n,1}(\alpha,\beta)$, where $P_{n,1}(\alpha,\beta)$ is the open rational homology "polydisc", whose moment image is shown in \cref{fig:atfrationalblowup_non_squeezing_theorem}.

Now we view the embedding $\iota$ as an embedding into the compactification $X_{n,1}(\alpha,\beta)$, i.e.\ consider, by abuse of notation, $B\subseteq X_{n,1}(\alpha,\beta)$. 
Note that this means that we possibly have to pick a larger $\beta$ in order to get admissible values $\alpha$ and $\beta$ for the compactification.\footnote{Recall that $X_{n,1}(\alpha,\beta)$ denotes the compactification introduced in \cref{subsec:compactification}.}
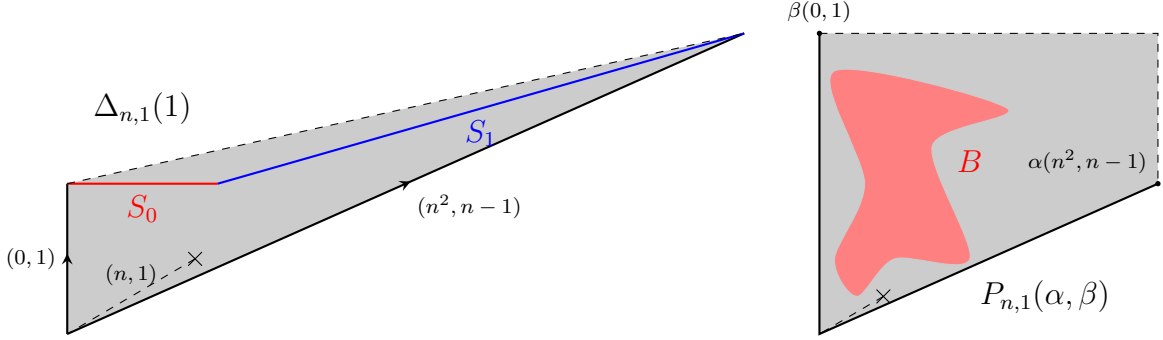
\begin{figure}[ht]
    \centering
        \begin{tikzpicture}
        \begin{scope}[shift={(-4.5,-2)}]
            \fill[opacity=0.2] 
                (0,2) -- (0,0) -- (9,4);
            \draw[thick, mid arrow] 
                (0,0) -- node[anchor=east] {\tiny $(0,1)$} 
                (0,2);
            \draw[thick, mid arrow] 
                (0,0) -- node[anchor=north west] {\tiny $(n^2,n-1)$} 
                (9,4);
            \node at (1,3) {$\Delta_{n,1}(1)$};
            \draw[thick, red]
                (0,2) -- node[anchor=north] {$S_0$} (2,2);
            \draw[thick, blue]
                (2,2) -- node[anchor=north] {$S_1$} (9,4);
            \draw[dashed] 
                (0,2) -- (9,4);
            \draw[dashed] 
                (1.7,1) node[cross] {} -- node[anchor=south] {\tiny $(n,1)$} (0,0);
        \end{scope}

        \begin{scope}[shift={(5.5,-2)}]
            \fill[opacity=0.2] (0,4) -- (0,0) -- (4.5,2) -- (4.5,4);
            \draw[thick] (0,4) -- (0,0) -- (4.5,2);
            \draw[dashed] (0.85,0.5) node[cross] {} -- (0,0);
            \draw[dashed] 
                (0,4) node[anchor=south] {\tiny$\beta(0,1)$} -- 
                (4.5,4) --
                (4.5,2) node[anchor=south east]{\tiny $\alpha(n^2,n-1)$};
            \node at (3,0.5) {$P_{n,1}(\alpha,\beta)$};
            \filldraw [fill=red!50,draw=black!20] plot [mark=none, smooth cycle] coordinates {(0.5,0.5) (1,1) (2,1) (1.5,2.5) (2.5,3) (0.2,3.5) (0.6,2) (0.2,1)};
            \node[red] at (2,2.3) {$B$};
            \fill (0,4) circle (1pt);
            \fill (4.5,2) circle (1pt);
        \end{scope}
        \end{tikzpicture}
    \caption{On the left the choice of sizes involved in the rational blow-up used during the proof of the rational non-squeezing theorem and on the right the rational homology "polydisc".}
    \label{fig:atfrationalblowup_non_squeezing_theorem}
\end{figure}
We now want to rationally blow-up this embedding.
Recall that performing the rational blow-up involves choosing the sizes of the chain of symplectic spheres $C_n$. 
We choose the sizes of the spheres in the chain as depicted in \cref{fig:atfrationalblowup_non_squeezing_theorem}. For details on why all these choices are valid compare \cref{Appendix:sizes_blow-up}. This means that for small enough $0<\epsilon$ we choose these areas to be
\begin{equation*}
    c_0=4-\sum_{j=2}^{n-1} \big(j(n+1)+1\big)\epsilon, \quad c_1=n-2 \quad \text{and} \quad c_i=\epsilon \text{ for } 2\leq i\leq n-2.
\end{equation*}
The visible pinwheel in $B$ is liminal, as it is disjoint from the spheres $\Sigma_+$ and $\Sigma_-$.\footnote{Recall that by \cref{Lemma_Compactification_S2S2_X1} the compactification $X_{n,1}(\alpha,\beta)$ is either $S^2 \times S^2$ or $X_1$.}
Therefore Corollary \ref{cor: mu_i periods} tells us that the manifold obtained via the blow-up contains an exceptional sphere of area
\begin{equation}\label{eq:mu_n_non-squeezing}
    \mu_{n}=\frac{1}{n^2}\Big((n+2)d_1-(n+1)d_2-\sum_{j=0}^{n-2}\big(1+j(n+1)\big)c_j\Big).
\end{equation}
Recall that we computed the sizes of the sphere $\Sigma_+$ and $\Sigma_-$ in \cref{lemma_computing_sizes_compactification} and \cref{remark:sizes_compactification}, i.e.\ we have that $\omega_{\alpha,\beta}(\Sigma_+)=(n+1)\beta + \alpha$ and that $\omega_{\alpha,\beta}(\Sigma_-)=\beta - (n-1)\alpha$. The positive real numbers $d_1$ and $d_2$ are the areas of the classes $W_1$ and $W_2$ respectively, which are related to the classes $[\widetilde{\Sigma}_+]$ and $[\widetilde{\Sigma}_-]$ via
\[W_1=[\widetilde{\Sigma}_+]\text{ and } W_2=[\widetilde{\Sigma}_+] + [\widetilde{\Sigma}_-],\]
thus giving
\[d_1=(n+1)\beta+\alpha\text{ and } d_2=(n+2)\beta-(n-2)\alpha.\]
Plugging these relations into Formula \ref{eq:mu_n_non-squeezing} for $\mu_n$ we get
\begin{equation*}
\mu_n=\alpha-1+\epsilon.
\end{equation*}
Since $\mu_n$ has to be positive and $\epsilon>0$ was arbitrary we obtain
\[\alpha\geq 1,\]
as claimed.
\end{proof}

\begin{remark}
    This strategy of proof also yields a proof of Gromov's non-squeezing result in dimension 4, which is the case $n=1$, and the case $n=2$, which will be carried out in \cite{BrEvHaSch25}.
\end{remark}

\appendix
\section{Solving for \texorpdfstring{$\mu_{n-1}$}{mun-1} and \texorpdfstring{$\mu_{n}$}{mun}}
\label{Appendix:solving_mu}

For the sake of completeness, we will explicitly perform the computation on which \cref{cor: mu_i periods} depends on. 
This is nothing more than elementary matrix operations. 
Let $d_i,c_r$ be the periods of the homology classes $D_i,S_r$ and let $h,\mu_i$ be the periods of the standard basis. Writing down the relation of $d_i,c_r$ to $h,\mu_i$ we have:

\begin{center}
\begin{tabular}{ c c c c c c c c c c}
  & $h$ & $\mu_1$ & $\mu_2$ & $\mu_3$ & $\cdots$ & $\mu_{n-3}$ & $\mu_{n-2}$ & $\mu_{n-1}$ & $\mu_{n}$ \\
 $c_{0}$ & $-2$ & $3$ & $-1$ & $-1$ & $\cdots$ & $-1$ & $-1$ & $0$ & $0$\\
 $c_{1}$ & $0$ & $0$ & $0$ & $0$ & $\cdots$ & $0$ & $1$ & $-1$ & $0$ \\ 
 $c_{2}$ & $0$ & $0$ & $0$ & $0$ & $\cdots$ & $1$& $-1$ & $0$ & $0$ \\
  
 $\vdots$ & $\vdots$ & $\vdots$ & $\vdots$ & $\vdots$ & & $\vdots$ & $\vdots$ & $\vdots$ & $\vdots$\\
 $c_{n-4}$ & $0$ & $0$ & $0$ & $1$ & $\cdots$ & $0$ & $0$ & $0$ & $0$ \\ 
 $c_{n-3}$ & $0$ & $0$ & $1$ & $-1$ & $\cdots$ & $0$ & $0$ & $0$ & $0$ \\ 
 $c_{n-2}$ & $1$ & $-1$ & $-1$ & $0$ & $\cdots$ & $0$ &$0$ & $0$ & $-1$\\
 $d_1$ & $n$ & $-(n-1)$ & $-1$ & $-1$ & $\cdots$ & $-1$ & $-1$ & $-1$ & $0$\\
 $d_2$ & $3$ & $-2$ & $0$ & $0$ & $\cdots$ & $0$ & $0$ & $0$ & $-1$
\end{tabular}
\end{center}

Substituting $c'_{n-3}=c_{n-3}$, $c'_{i}=c_i+c'_{i+1}$ for $1\leq i \leq n-4$, $c'_0=c_0-\sum_{i=1}^{n-3}c'_{i}+2c_{n-2}$ and $c_{n-2}'=c_{n-2}+c'_0$ we get

\begin{center}
\begin{tabular}{ c c c c c c c c c c}
  & $h$ & $\mu_1$ & $\mu_2$ & $\mu_3$ & $\cdots$ & $\mu_{n-3}$ & $\mu_{n-2}$ & $\mu_{n-1}$ & $\mu_{n}$ \\
 $c'_{0}$ & $0$ & $1$ & $-n+1$ & $0$ & $\cdots$ & $0$ & $0$ & $0$ & $-2$\\
 $c'_{1}$ & $0$ & $0$ & $1$ & $0$ & $\cdots$ & $0$ & $0$ & $-1$ & $0$ \\ 
 $c'_{2}$ & $0$ & $0$ & $1$ & $0$ & $\cdots$ & $0$& $-1$ & $0$ & $0$ \\
  
 $\vdots$ & $\vdots$ & $\vdots$ & $\vdots$ & $\vdots$ & & $\vdots$ & $\vdots$ & $\vdots$ & $\vdots$\\
 $c'_{n-4}$ & $0$ & $0$ & $1$ & $0$ & $\cdots$ & $0$ & $0$ & $0$ & $0$ \\ 
 $c'_{n-3}$ & $0$ & $0$ & $1$ & $-1$ & $\cdots$ & $0$ & $0$ & $0$ & $0$ \\ 
 $c'_{n-2}$ & $1$ & $0$ & $-n$ & $0$ & $\cdots$ & $0$ &$0$ & $0$ & $-3$\\
 $d_1$ & $n$ & $-(n-1)$ & $-1$ & $-1$ & $\cdots$ & $-1$ & $-1$ & $-1$ & $0$\\
 $d_2$ & $3$ & $-2$ & $0$ & $0$ & $\cdots$ & $0$ & $0$ & $0$ & $-1$
\end{tabular}
\end{center}

Finally, for $d'_1=d_1+(n-1)c'_0-(n+2)c'_1-(\sum_{i=2}^{n-3}c'_i)-nc'_{n-2}$
and $d'_2=d_2+2c'_0-(n+2)c'_1-3c'_{n-2}$ we get

\begin{center}
\begin{tabular}{ c c c c c c c c c c}
 0 & $h$ & $\mu_1$ & $\mu_2$ & $\mu_3$ & $\cdots$ & $\mu_{n-3}$ & $\mu_{n-2}$ & $\mu_{n-1}$ & $\mu_{n}$ \\
 $c'_{0}$ & $0$ & $1$ & $-n+1$ & $0$ & $\cdots$ & $0$ & $0$ & $0$ & $-2$\\
 $c'_{1}$ & $0$ & $0$ & $1$ & $0$ & $\cdots$ & $0$ & $0$ & $-1$ & $0$ \\ 
 $c'_{2}$ & $0$ & $0$ & $1$ & $0$ & $\cdots$ & $0$& $-1$ & $0$ & $0$ \\
  
 $\vdots$ & $\vdots$ & $\vdots$ & $\vdots$ & $\vdots$ & & $\vdots$ & $\vdots$ & $\vdots$ & $\vdots$\\
 $c'_{n-4}$ & $0$ & $0$ & $1$ & $0$ & $\cdots$ & $0$ & $0$ & $0$ & $0$ \\ 
 $c'_{n-3}$ & $0$ & $0$ & $1$ & $-1$ & $\cdots$ & $0$ & $0$ & $0$ & $0$ \\ 
 $c'_{n-2}$ & $1$ & $0$ & $-n$ & $0$ & $\cdots$ & $0$ &$0$ & $0$ & $-3$\\
 $d'_1$ & $0$ & $0$ & $0$ & $0$ & $\cdots$ & $0$ & $0$ & $n+1$ & $n+2$\\
 $d'_2$ & $0$ & $0$ & $0$ & $0$ & $\cdots$ & $0$ & $0$ & $n+2$ & $4$
\end{tabular}
\end{center}

This implies that $\mu_{n-1}=\frac{-4d'_1+(n+2)d'_2}{n^2}$ and $\mu_{n}=\frac{(n+2)d'_1-(n+1)d'_2}{n^2}$, which becomes

\begin{align*}
\label{eq:equation_for_mun}
    \mu_{n}&=\frac{1}{n^2}\Big((n+2)d_1-(n+1)d_2-\sum_{i=0}^{n-2}\big(1+i(n+1)\big)c_i\Big)\\
    \mu_{n-1}&=\frac{1}{n^2}\Big((n+2)d_2-4d_1-\Big((n-2)c_0+(n+2)\sum_{i=1}^{n-2}(n-(i+1))c_i\Big)\Big)
\end{align*}

\section{Blowing up while taking account of sizes}
\label{Appendix:sizes_blow-up}

Assume that $n\geq 3$ and consider the rational homology ball $B_{n,1}(1)$. 
We claim that for $\epsilon>0$ small enough there is a choice of rationally blowing $B_{n,1}(1)$ that introduces a chain of spheres $C_n=(S_0,\ldots,S_{n-2})$, such that the area of the spheres are given by 
\begin{equation}\label{eq:app_sizes}
    c_0=4-\sum_{j=2}^{n-1} \big(j(n+1)+1\big)\epsilon, \quad c_1=n-2 \quad \text{and} \quad c_i=\epsilon \text{ for } 2\leq i\leq n-2.
\end{equation}

To see this recall that the rational blow-up in the almost toric interpretation is done by replacing a standard neighborhood of the pinwheel in $B_{n,1}$ with the minimal resolution of the quotient singularity $\frac{1}{n}(1,n-1)$, as described in \cite[Chapter 9]{Ev23:Book} and \cref{Section_Blowing_Up_Pinwheels}. 
The geometry of the rational blow-up, i.e.\ all the vectors involved in the symplectic cutting procedure of the minimal resolution, is shown in \cref{fig:geometry_of_minimal resolution}.
\begin{figure}[ht]
    \centering
        \begin{tikzpicture}
        \begin{scope}
            \fill[opacity=0.2] 
                (0,0) -- (6,0) -- (11,3) -- (0,3);
            \draw[thick, mid arrow] 
                (0,0) -- node[anchor=east]{\tiny $(0,1)$} (0,3);
            \draw[thick, mid arrow]
                (6,0) -- node[anchor=north west]{\tiny $(n^2,n-1)$} (11,3);
            \node at (1,2) {\large $\Delta_{n,1}$};
            \draw[thick, mid arrow]
                (0,0.5) -- node[anchor=south]{$S_0$} node[anchor=north]{\tiny $(1,0)$} (2,0.5);
            \draw[thick, mid arrow]
                (2,0.5) -- node[anchor=south]{$S_1$} node[below right]{\tiny $(n+2,1)$} (5,1);
            \draw[thick, dotted]
                (5,1) -- (7.5,1.7);
            \draw[thick]
                (7.5,1.7) -- node[anchor=south]{$S_{n-2}$}(11,3);
            \draw[dashed] 
                (0,3) -- (11,3);
            \draw[dashed]
                (0,0) -- (6,0);
        \end{scope}
        \end{tikzpicture}
    \caption{The geometry of the rational blow-up/blow-down.}
    \label{fig:geometry_of_minimal resolution}
\end{figure}
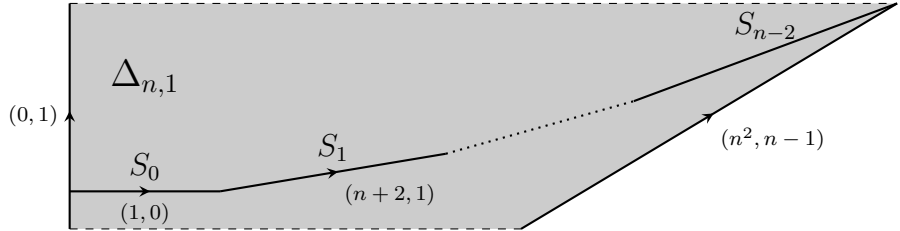
The primitive vector associated to the sphere $S_j$ for $0\leq j\leq n-2$ is given by $(j(n+1)+1,j)$.

Now, assume that we want to blow-up the rational homology ball $B_{n,1}(1)$. 
To do this in a manner that takes the quantitative aspects into account we need to specify the lengths of the primitive vectors associated to the chain of spheres $C_n$, i.e.\ we need to "fit" the configuration of line segments depicted in \cref{fig:geometry_of_minimal resolution} between the primitive vectors $(0,1)$ and $(n^2,n-1)$. 
To see that a choice of sizes as claimed in \cref{eq:app_sizes} is possible we fit the configuration of line segments so that the line segment associated to the spheres $S_{n-2}$ terminates on the right in $(n^2-\epsilon n^2,n-1-\epsilon(n-1))$. 
Then choosing the rest of the sizes as claimed in \cref{eq:app_sizes}, which means giving $S_{n-2}$ size $\epsilon$, then giving $S_{n-3}$ size $\epsilon$, etc., shows that the line segment associated to $S_0$ terminates in $(0,1-(\sum_{j=2}^{n-1}j)\epsilon)$ and the configuration fits into the almost toric base diagram of $B_{n,1}(1)$.

\section{The orthogonal complement of a piwnheel}
\label{Appendix:orthogonal complement}

In this appendix we investigate the meaning of the "homological complement" of a Lagrangian pinwheel. 
This is crucially needed in \cref{sec:obstruction}. 
Assume for the rest of this appendix that $(X,\omega)$ is a closed symplectic $4$-manifold such that $H_1(X;\mathbb{Z})=0$ and that $L$ is an $L_{p,q}$-pinwheel therein. 

Here is a definition of the "orthogonal complement" of the class $[L] \in H_2(X;\mathbb{Z}_p)$ defined by the $L_{p,q}$-pinwheel.

\begin{definition}\label{def:orth_comp}
    Define the \textit{orthogonal complement} of $[L]$ by
    \[[L]^\perp\vcentcolon=\left\{C\in H_2(X;\mathbb{Z})\hspace{0.1cm}\middle\vert\hspace{0.1cm} \widetilde{C}\cdot [L]\equiv 0\mod p \right\},\]
    where $\widetilde{C}$ is the image of $C$ under the natural map $H_2(X;\mathbb{Z}) \to H_2(X;\mathbb{Z}_p)$.
\end{definition}

By our \ref{def:Lag_pinwheels} of embeddings of Lagrangian pinwheels the pinwheel $L$ is actually coming from an embedding $B_{p,q} \hookrightarrow (X,\omega)$ and we will, by abuse of notation, write $L\subseteq B_{p,q}\subseteq X$. 
Denote the embedding of the complement by $\iota:X\setminus B_{p,q} \hookrightarrow X$. 
Recall that for the case that $(p,q)=(2,1)$, i.e.\ the case of embedding a Lagrangian $\mathbb{R}P^2$, \cite[Lemma 4.10]{BoLiWu13} proves that $[L]^\perp$ is contained in the image of $H_2(X\setminus B_{p,q};\mathbb{Z})\xhookrightarrow{\iota_*}H_2(X;\mathbb{Z})$ and that therefore this image is actually equal to $[L]^\perp$. 
The proof given by Borman, Li and Wu in \cite[Lemma 4.10.]{BoLiWu13} is of topological nature, relying on the fact that in the Lagrangian $\mathbb{R}P^2$ case the embeddings are actually honest embeddings of submanifolds. 
Compare also the discussion around \cite[Lemma 2.2.1]{SmSh20}, where this fact is discussed in a more algebraic manner.

This discussion motivates the following definition.

\begin{definition}\label{def:hom_comp}
    Define the \textit{homological complement} of $L$ by
    \[[L]^c \vcentcolon =\text{im}\left(H_2(X\setminus B_{p,q};\mathbb{Z})\xhookrightarrow{\iota_*}H_2(X;\mathbb{Z})\right).\]
\end{definition}

The goal of this appendix is to prove that \cref{def:orth_comp} and \cref{def:hom_comp} define the same object, i.e.\ that $[L]^c=[L]^\perp$. 
We will give a more algebraically flavored proof, leveraging the fact that $[L]^c$ and $[L]^\perp$ have the same index, i.e.\ the order of the quotient lattices satisfy 
\[|H_2(X;\mathbb{Z})/[L]^c|=|H_2(X;\mathbb{Z})/[L]^\perp|=p<\infty.\]

We start by computing the index of $[L]^c$, which was done in a slightly different context in \cite[Lemma 2.16]{EvSm18}. 
We quickly reproduce the argument for the convenience of the reader.

\begin{lemma}\emph{(\cite[Lemma 2.16 (a) and (b)]{EvSm18})}
Let $(X,\omega)$ be a closed symplectic $4$-manifold with $H_1(X;\mathbb{Z})=0$ and let $L$ be an $L_{p,q}$-pinwheel therein. 
Then $[L]^c$ has index $p$ and, in particular, has full rank since $p<\infty$ is finite.
\end{lemma}

\begin{proof}
Denote by $B$ the rational homology ball $B_{p,q}$ associated to the $L_{p,q}$-pinwheel $L$ and let $V$ be its complement, sharing the same boundary $\Sigma$, which is diffeomorphic to the lens spaces $L(p^2,pq-1)$. 
Then, using that $H_1(\Sigma;\mathbb{Z})=\mathbb{Z}_{p^2}$ and $ H_2(\Sigma;\mathbb{Z})=0$, Mayer--Vietoris gives
\[
\begin{tikzcd}
0 \arrow[r] & H_2(V;\mathbb{Z}) \arrow[r, "\alpha"] & H_2(X;\mathbb{Z}) \arrow[r, "\beta"] & \mathbb{Z}_{p^2} \arrow[r, "\gamma"] & \mathbb{Z}_{p}\oplus H_1(V;\mathbb{Z}) \arrow[r] & 0
\end{tikzcd}
.\]
Note that $H_1(V;\mathbb{Z})$ is trivial, which follows from the fact that surjectivity of $\gamma$ implies that $\mathbb{Z}_p\oplus H_1(V;\mathbb{Z})$ is cyclic, thus the order of $H_1(V;\mathbb{Z})$ must be coprime to $p$, but then, by surjectivity, it should also divide $p^2$ implying that it is $1$.
Moreover, the index of $[L]^c$ is exactly the order of $\text{coker}(\alpha)$ which by exactness of the sequence is $\text{im}(\beta)=\text{ker}(\gamma)$. 
However, $\text{ker}(\gamma)=\mathbb{Z}_p$ since $\gamma$ is just reduction modulo $p$.
\end{proof}

\begin{lemma}
Let $(X,\omega)$ be a closed symplectic $4$-manifold with $H_1(X;\mathbb{Z})=0$ and let $L$ be an $L_{p,q}$-pinwheel therein.
Then, $[L]^\perp$ has index $p$ and, in particular, has full rank since $p<\infty$ is finite.
\end{lemma}

\begin{proof}
Consider the functional $F_{[L]}: H_2(X;\mathbb{Z})\rightarrow \mathbb{Z}_p$ given by $F_{[L]}(C)=\widetilde{C}\cdot [L]\mod p$, where $\widetilde{C}$ is the image of $C$ under the natural map $H_2(X;\mathbb{Z}) \to H_2(X;\mathbb{Z}_p)$. 
Since $[L]^\perp=\text{ker}(F_{[L]})$, we only have to show that $F_{[L]}$ is surjective. Quite pleasantly, this follows from \cref{Lag_pinwheels_self_int}, since $F_{[L]}(\widetilde{[L]})=-1 \mod p $, where $\widetilde{[L]}$ is some integral lift of $[L]$.
\end{proof}

Since $[L]^c\subseteq [L]^\perp$ and the two subgroups have the same index, we arrive at the desired conclusion:

\begin{theorem}\label{thrm:pinwheel_hom/orth_comp}
Let $(X,\omega)$ be a closed symplectic $4$-manifold with $H_1(X;\mathbb{Z})=0$ and let $L$ be an $L_{p,q}$-pinwheel therein. 
Then, a homology class $C\in H_2(X;\mathbb{Z})$ is represented by a smooth submanifold disjoint from $L$ if and only if 
\[\widetilde{C}\cdot [L]=0\mod p,\]
where $\widetilde{C}$ is the image of $C$ under the natural map $H_2(X;\mathbb{Z}) \to H_2(X;\mathbb{Z}_p)$, i.e.\ $[L]^c=[L]^\perp$.
\end{theorem}

\emergencystretch=1em

\printbibliography
\end{document}